\documentclass[pdftex]{amsart}
\usepackage{amsmath, amssymb, amscd}
\usepackage{amsthm}
\usepackage{mathrsfs}
\usepackage{mathdots}
\usepackage{url}
\usepackage[all]{xy}
\usepackage{tikz-cd}
\usepackage[top=30truemm,bottom=30truemm,left=25truemm,right=25truemm]{geometry}
\usepackage{here}

\newcommand{\aaf}{\mathbb A_F}

\newcommand{\aae}{\mathbb A_E}

\newcommand{\res}{\mathrm{Res}_{E/F}}
\newcommand{\sll}{{\mathrm {SL}}}

\newcommand{\gll}{{\mathrm {GL}}}
\newcommand{\zz}{{\mathbb Z}}
\newcommand{\rr}{{\mathbb R}}

\newcommand{\cc}{{\mathbb C}}

\newcommand{\bu}{{\mathrm U}}
\newcommand{\uu}{{\mathrm U (1, 1)}}

\newcommand{\uuu}{{\mathrm U (3)}}

\newcommand{\uuuu}{{\mathrm U (2, 2)}}

\newcommand{\guu}{{\mathrm {GU} (1, 1)}}

\newcommand{\guuu}{{\mathrm {GU} (3)}}

\newcommand{\guuuu}{{\mathrm {GU} (2, 2)}}

\newcommand{\go}{\mathrm {GO} (4, 2)}

\newcommand{\gso}{\mathrm {GSO} (4, 2)}

\newcommand{\ep}{\epsilon_{E/F}}
\newcommand{\inj}{\hookrightarrow}
\newcommand{\surj}{\twoheadrightarrow}
\newcommand{\map}{\rightarrow}
\newcommand{\bij}{\xrightarrow{\sim}}

\newcommand{\ind}{\mathrm{Ind}}
\newcommand{\baz}{^\times}

\theoremstyle{definition}
\newtheorem{theo}{Theorem}[section]
\newtheorem{defi}[theo]{Definition}
\newtheorem{prop}[theo]{Proposition}
\newtheorem{lemm}[theo]{Lemma}
\newtheorem{cor}[theo]{Corollary}
\newtheorem{rem}[theo]{Remark}

\newtheorem{rmk}[theo]{Remark}
\newtheorem*{q}{\underline {Questions}}

\begin{document}

\title{On Miyawaki lifts with respect to Hermitian Maass lifts}

\author{Nozomi Ito}
\address{Department of Mathematics,
Kyoto University,
Kitashirakawa Oiwake-cho, Sakyo-ku,
Kyoto 606-8502, Japan}
\email{nozomi@math.kyoto-u.ac.jp}
\maketitle

\begin{abstract}
In this paper, we discuss the representation-theoretical Miyawaki lift for unitary groups in terms of the endoscopic classification. We give an explicit determination of Miyawaki lifts for $\bu(1)$ and $\bu(3)$ with respect to Hermitian Maass lifts, namely Ikeda lifts for $\bu(4)$.
\end{abstract}

\setcounter{tocdepth}{1}
\tableofcontents
\section{Introduction}
In \cite{ikeda2006pullback}, Ikeda gave a lifting of Siegel modular forms using diagonal restrictions of Ikeda lifts \cite{ikeike} as kernel functions and described the L-functions of the liftings under the assumption that they do not vanish.
The lifting is called the Miyawaki lift.
Today, there exist some analogues and generalizations of the Miyawaki lift \cite{atobe2018miyawaki}, \cite{kim2018miyawaki}, \cite{attarticle}.

The main reason why the Miyawaki lift works is that almost all local components of automorphic representations generated by Ikeda lifts are (quotients of) degenerate principle series.
Since diagonal restrictions of degenerate principle series have simple branching laws on unramified representations (see \cite[Proposition 3.1]{ikeda2006pullback} or \cite[Lemma 2.2]{gt1}), we can determine the near equivalence classes of Miyawaki lifts.

Based on above, we can redefine the Miyawaki lift representation-theoretically and relatively axiomatically by using the endoscopic classification (see \cite{arthur2011endoscopic}, \cite{mok2015endoscopic}, \cite{kltarticle}).
In this paper, we treat this in the case of unitary groups mainly (see \S 3 for more details).

Let $F$ be a number field and $E$ a quadratic extension of $F$.
Let $V=V_1\perp V_2$ be a $2n$-dimensional nondegenerate Hermitian space over $E$ which is decomposed into an $m$-dimensional nondegenerate subspace $V_1$ and a $(2n-m)$-dimensional nondegenerate subspace $V_2$, where $m\leq n$.
Then there is a natural injection
$$i:\bu(V_1)\times\bu(V_2)\inj\bu(V),$$
where $\bu(V')$ is the isometry group of an Hermitian space $V'$.
Let $\Pi$ and $\pi$ be discrete automorphic representations of $\bu(V)(\aaf)$ and $\bu(V_1)(\aaf)$, respectively, where $\aaf$ is the adele ring of $F$.
We assume that the A-parameter of $\Pi$ is equal to 
$$\psi:=\phi_2[n]=\phi_2\boxtimes[n],$$
where $\phi_2$ is a generic A-parameter of $\gll_2(\aae)$ with some conditions (see \S 3) and $[n]$ is the $n$-dimensional irreducible algebraic representation of $\sll_2(\cc)$.
By this assumption, almost all local components of $\Pi$ are degenerate principle series.
In addition, we assume that 
$$\mathcal M_\psi(\psi'):=\psi'^\vee\boxplus\phi_2[n-m]$$
is a discrete A-parameter of $\bu(V_2)$, where $\psi'$ is the A-parameter of $\pi$.
For $\mathcal F \in \Pi$ and $f\in\pi$, we define $\mathcal M_{\mathcal F}(f)$ by
$$\mathcal M_{\mathcal F}(f)(g_2) :=\langle \mathcal F\circ i(\cdot,g_2), f\rangle_{\bu (V_1)},\ g_2\in \bu(V_2)(\aaf)$$
as long as it converges, where $\langle\cdot,\cdot\rangle_{\bu (V_1)}$ is the Peterson inner product on $\bu (V_1)(\aaf)$.
If $\mathcal M_{\mathcal F}(f)$  converges for all $\mathcal F\in\Pi$ and $f\in \pi$, we define $\mathcal M_{\Pi}(\pi)$ as the representation of $\bu(V_2)(\aaf)$ generated by all $\mathcal M_{\mathcal F}(f)$.
Then, thanks to the above assumptions, $\mathcal M_{\Pi}(\pi)$ is discrete and all irreducible constituent of $\mathcal M_{\Pi}(\pi)$ have the same A-parameter $\mathcal M_\psi(\psi')$ if $\mathcal M_{\Pi}(\pi)$ is square integrable.

Following \cite{ikeda2006pullback}, we call $\mathcal M_{\Pi}(\pi)$ the Miyawaki lift of $\pi$ with respect to $\Pi$.
Our interest is to determine $\mathcal M_{\Pi}(\pi)$, namely to answer the following three questions:
\begin{q}
\begin{enumerate}
\item When is $\mathcal M_{\Pi}(\pi)$ square integrable?
\item When is $\mathcal M_{\Pi}(\pi)$ nonzero?
\item What are irreducible constituents of $\mathcal M_{\Pi}(\pi)$ (in the sense of the endoscopic classification)?
\end{enumerate}
\end{q}

 Our main result is to answer the above questions in a low rank case.
\begin{theo}[Theorem \ref{T:main}]
If $V=\mathbb H_E^2$ is the direct sum of two hyperbolic plane and $m=1$, then $\mathcal M_{\Pi}(\pi)$ is zero or irreducible cuspidal automorphic representations
.
Moreover, the equivalence classes of $\mathcal M_{\Pi}(\pi)$ can be determined explicitly.
\end{theo}

This paper is organized as follows. In \S \ref{S:ec}, we review the endoscopic classification for quasisplit unitary groups.  In \S \ref{S:im}, we reformulate the Ikeda lift and the Miyawaki lift in sense of the endoscopic classification. In \S \ref{S:hm}, we show that Hermitian Maass lifts, namely Ikeda lifts on $\bu(\mathbb H_E^2)(\aaf)$,  can be described by theta lifts. 
In \S \ref{S:mt}, we prove the main theorem.
\subsection*{Acknowledgment}
The author would like to thank my adviser Prof. Atsushi Ichino.
Thanks to his helpful advice, I could complete this work.
\subsection*{Notation}
Throughout the paper we denote by $F$ a global or local field with characteristic 0 and by $E$ a quadratic extension (resp. etale quadratic algebra) of $F$ if $F$ is a global (resp. local) field.
We denote by $c$ the nontrivial $F$-automorphism of $E$.
We fix a tracezero element $\kappa$ of $E^\times$ and let $d=\kappa^2\in F^\times$.

We suppose that $F$ be a local field.
Then, we denote by $|\cdot|_{F}$ the normalized absolute value on $F\baz$ and we define $|\cdot|_{E}$ by $|z|_{E}=|zz^c|_{F}$ ($z\in E\baz$).

We suppose that $F$ is a global field.
Then, We denote by $F_v$ the completion of $F$ with respect to a place $v$ of $F$ and by $E_v$ the tensor product $F_v\otimes_F E$.
We denote the rings of adeles of $F$ and $E$ by $\aaf$ and $\aae$, respectively.
Furthermore, we denote the idele norm $\otimes_v|\cdot|_{F_v}$ on $\aaf\baz$ by $|\cdot|_{\aaf}$ and we define $|\cdot|_{\aae}$ by $|z|_{\aae}=|zz^c|_{\aaf}$ ($z\in \aae\baz$).

Let $G$ be an algebraic group over $F$.
If $F$ is a local field, we often identify an algebraic group $G$ over $F$ with the group of $F$-rational points $G(F)$ of $G$.
If $F$ is a global field and $v$ is a place of $F$, we define a group $G_v$ over $F_v$ by $G_v(R)=G(R)$ for any $F_v$-algebra $R$.

We denote by $\psi_F$ a fixed nontrivial additive character of $\aaf/F$ (resp. $F$) if $F$ is a global (resp. local) field. 

For $n\in\zz_{>0}$, let 
$$J_n=\begin{pmatrix}
 &&1 \\
 &\iddots& \\
 1&&
\end{pmatrix}\in\gll_n(\zz).$$
We define the quasisplit (similitude) unitary group by
\begin{align*}
\mathrm {GU}(n)=\mathrm {GU}_{E/F}(n)&=
\{g \in \res\gll_{n} \ | \ {}^t\! g^cJ_{n}g=\nu_{n}(g)J_{n}, \nu_{n}(g)\in \gll_1 \},\\
\mathrm {U}(n)=\mathrm {U}_{E/F}(n)&=\mathrm {ker}(\nu_{n}).
\end{align*}
For convenience, let $\mathrm {GU}(0)=\gll_1, \nu_0={\rm id}_{\mathrm {GU}(0)},$ and $\mathrm {U}(0)=\mathrm {ker}(\nu_{0})=\{1\}.$
We often denote $\nu_{n}$ by $\nu$ for short.
We denote by $B'_{n}$ the Borel subgroup of $\mathrm{GU}(n)$ which consists of upper triangular matrices and let $B'_{n}=M'_{n}N_{n}$ be its Levi decomposition.
For $\mathrm{U}(n)$, let $B_{n}=B'_{n}\cap\mathrm{U}(n)=M_{n}N_{n}$, similarly. 

%

When $F$ is a local field, we denote by $\tau_1\times\tau_2\times\cdots\times\tau_{k}\rtimes\pi_0$ a normalized parabolic induction
$${\rm Ind}_{P'}^{\mathrm {GU}(n)}\tau_1\boxtimes\tau_2\boxtimes\cdots\boxtimes\tau_k\boxtimes\pi_0,$$
where
\begin{itemize}
\item $l_1, \dots ,l_k$ are positive integers such that $n_0=n-2\sum_i l_i\geq 0$,
\item $P'$ is the parabolic subgroup of $\bu(n)$ containing $B'_{n}$ whose Levi subgroup is isomorphic to $\prod_i \res\gll_{l_i}\times \mathrm {GU}(n_0)$.
\item $\tau_1,\dots,\tau_k,$ and $\pi_0$ are representations of $\res\gll_{l_1},\dots,\res\gll_{l_k},$ and $\mathrm{GU}(n_0)$, respectively,
\item $\tau_1\boxtimes\tau_2\boxtimes\cdots\boxtimes\tau_k\boxtimes\pi_0$ is regarded as a representation of $P'$ through the map
\begin{alignat*}{3}
P'\quad\quad\quad\quad\quad\quad&\surj \prod_i \res\gll_{l_i}\times \mathrm {GU}(n_0)\\
\begin{pmatrix}
g_1&*&*&*&*&*&*\\
&\ddots&*&*&*&*&*\\
&&g_k&*&*&*&*\\
&&&g_0&*&*&*\\
&&&&*&*&*\\
&&&&&\ddots&*\\
&&&&&&*
\end{pmatrix}
&\mapsto(g_1,\dots,g_k,g_0)
\end{alignat*}
(if $n_0=0,$ interpret $g_0$ appropriately).
\end{itemize}
For $\mathrm {U}(n)$, we define $\tau_1\times\tau_2\times\cdots\times\tau_{k}\rtimes\pi_0$ similarly.

\section{Endoscopic classification for quasisplit unitary groups}\label{S:ec}
The endoscopic classification is a classifying method for automorphic representations of some classical groups established by Arthur \cite{arthur2011endoscopic}.
The theory gives a multiplicity formula of automorphic discrete representations using global A-parameters and characters of some (component) groups which depend on the localizations of A-parameters.

In this section, we summarize some properties of the endoscopic classification for quasisplit unitary groups proved by Mok \cite{mok2015endoscopic}.
\subsection{Local theory}
Firstly we make a survey of the local endoscopic classification. Let $F$ be a local field of characteristic 0.
\subsubsection{Local parameters}
Let $W_{F}$ be the Weil group of $F$ and let
\begin{align*}
L_{F}=
\begin{cases}
W_{F}\times\sll_2(\cc) \ &\mbox{if $F$ is nonarchimedean};\\
W_{F} \ &\mbox{if $F$ is archimedean}.
\end{cases}
\end{align*}
We denote by $\Phi(\bu(n))$ the set of conjugacy classes of L-parameters i.e. $W_F$-homomorphisms
$$\phi:L_F\map \!^L\bu(n)= \widehat{\bu(n)}\rtimes W_F$$
where 
\begin{itemize}
\item $\phi|_{W_F}$ is smooth,
\item for any continuous representation 
$$r:\!^L\bu(n)\map \gll_M(\cc)$$
which is algebraic on $\widehat{\bu(n)}$, $r\circ\phi(w,1)$ is a semisimple element of $\gll_M(\cc)$ for any $w\in W_F$,
\item if $F$ is nonarchimedean, then the homomorphism
$$\phi|_{\sll_2(\cc)}: \sll_2(\cc)\map \widehat{\bu(n)}$$ 
is algebraic.
\end{itemize}
Here, $\widehat{\bu(n)}\simeq\gll_n(\cc)$ is the Langlands dual group of ${\bu(n)}$.
We denote by $\Phi_{\rm bdd}(\bu(n))$ the subset of $\Phi(\bu(n))$ which consists of parameters $\phi\in\Phi(\bu(n))$ such that
the image of $\phi(W_F)$ in $\widehat{\bu(n)}$ is bounded. 

Let $\psi:L_F\times\sll_2(\cc)\map\!^L\bu(n)$ be a homomorphism.
We say that $\psi$ is an A-parameter of $\bu(n)$ if $\psi|_{L_F}\in\Phi(\bu(n))$ and $\psi|_{\sll_2(\cc)}$ is algebraic.
We denote by $\Psi^+(\bu(n))$ the set of conjugacy classes of A-parameters of $\bu(n)$ and by $\Psi(\bu(n))$ the subset of $\Psi^+(\bu(n))$ consisting $\psi\in\Psi^+(\bu(n))$ such that $\psi|_{L_F}\in\Phi_{\rm bdd}(\bu(n))$.
We identify $\psi \in\Psi^+(\bu(n))$ such that $\psi|_{\sll_2(\cc)}=1$ with $\psi|_{L_F} \in\Phi(\bu(n))$.
We regard $\Phi(\bu(n))$ (resp. $\Phi_{\rm bdd}(\bu(n))$) is a subset of $\Psi^+(\bu(n))$ (resp. $\Psi(\bu(n))$) by this identification.
\subsubsection{Conjugate selfdual parameters}
Let $E/F$ be a quadratic field extension.
We define $W_E$ and $L_E$ similarly to $W_F$ and $L_F$.
We fix $w_c\in W_F\setminus W_E$.
Let $\psi'$ be an $N$-dimensional semisimple complex representations of $L_{E}\times\sll_2(\cc)$ which satisfies that $\psi'|_{W_E}$ is smooth and the restriction of $\psi'$ to each $\sll_2(\cc)$-factor is algebraic.
We define $(\psi')^{c}$ by $(\psi')^{c}(w)=\psi'(w_cww_c^{-1})$ for $w\in L_{E}\times\sll_2(\cc)$.
For $b=\pm1$, we say that $\psi'$ is conjugate selfdual with parity $b$ if there is a nondegenerate bilinear form $B: \cc^n\times\cc^n\map\cc$ which satisfies that
\begin{align*} B(\psi'(w)x,(\psi')^{c}(w)y)&=B(x,y),\\
B(y,x)&=bB(x,\psi'(w_c^2)y)
\end{align*}
for any $x, y\in\cc^n$ and $w\in L_{E}\times\sll_2(\cc)$.

Since the action of $W_E$ on $\widehat{\bu(n)}$ is trivial, we regard $\psi|_{L_E\times\sll_2(\cc)}$ as an $n$-dimensional complex  representation of $L_E\times\sll_2(\cc)$ for $\psi\in \Psi^+(\bu(n))$.
Then the map $\psi\mapsto\psi|_{L_E\times\sll_2(\cc)}$ defines a bijection between conjugacy classes of A-parameters of $\bu(n)$ and conjugacy classes of conjugate selfdual $n$-dimensional representations of $L_E\times\sll_2(\cc)$ with parity $(-1)^{n-1}$.
Thus we often identify these two forms and use each of them depending on the context.
 
\subsubsection{Component groups}
For $\psi \in \Psi^+(\bu(n))$, let $S_{\psi}$ be the set of elements of $\widehat{\bu(n)}$ which commute with each element of the image of $\psi$.
We put $\mathcal S_{\psi}=S_{\psi}/S_{\psi}^0$ and $\mathcal {\overline S}_{\psi}=S_{\psi}/\{\pm 1_n S_{\psi}^0\}$, where $S_{\psi}^0$ is the connected component of $S_{\psi}$.
We note that if $E=F\times F$, then $\mathcal  S_{\psi}$
is trivial. If not, we can decompose $\psi$ as a direct sum
$$\psi=\left(\bigoplus_{i=1}^lm_i\psi_i\right)\oplus(\xi\oplus\xi^*),$$
where
\begin{itemize}
\item $\psi_i$ is a conjugate selfdual irreducible representation with parity $(-1)^{n-1}$,
\item if $\psi_i=\psi_j$ then $i=j$,
\item $m_i\in\zz_{>0}$ is the multiplicity of $\psi_i$ in $\psi$,
\item $\xi$ and $\ \xi^*=(\xi^c)^\vee$ are representations of $L_E\times\sll_2(\cc)$ which do not contain any conjugate selfdual irreducible subrepresentations with parity $(-1)^{n-1}$. 
\end{itemize}

Then, we have
$$\mathcal S_{\psi}\simeq(\zz/2\zz)^l.$$
 and 
\begin{align*}
\overline {\mathcal  S}_{\psi}\simeq
\begin{cases}
(\zz/2\zz)^l \ &\mbox{if every $m_i$ are even,}\\
(\zz/2\zz)^{l-1} &\mbox{otherwise.}
\end{cases}
\end{align*}
\subsubsection{Associated L-parameters}
Let $F'=E$ (resp. $F'=F$) if $E$ is a field (resp. $E=F\times F$). Then, for an $n$-dimensional representation $\psi$ of $L_{F'}\times\sll_2(\cc)$, we define an associated L-parameter $\phi_{\psi}$ of $\psi$ by $$\phi_{\psi}(w)=\psi(w,\mathrm {diag}(|w|^{\frac{1}{2}}_{F'},|w|^{-\frac{1}{2}}_{F'})),$$
where we regard $|\cdot|_{F'}$ as a character of $L_{F'}$ by $L_{F'}\surj W_{F'}\surj W_{F'}^{\rm ab}\simeq F'^\times$.
\subsubsection{The local endoscopic classification}
We fix a Whittaker datum  $\mathfrak w=(B,\lambda)$ of $\bu(n)$, where $B$ is a Borel subgroup of $\bu(n)$ and $\lambda$ is a nondegenerate character of the unipotent radical $N$ of $B$. 
The local endoscopic classification says that each $\psi \in \Psi(\bu(n))$ associates a finite multi-set $\Pi_{\psi}$ of irreducible unitary representations of $\bu(n)$ and there exists a map 
$$J_{\psi}=J_{\psi,\mathfrak w}: \Pi_{\psi} \map \hat{\overline {\mathcal  S}}_{\psi}=\{\eta\in\hat{{\mathcal  S}}_{\psi} \ | \ \eta(-1_n S_\psi^0)=1 \}$$
which depends on $\mathfrak w$.
We make a few remarks: 
\begin{itemize}
\item If $E=F\times F$, then $\Pi_{\psi}=\{\tau_{\phi_{\psi}}\}$, where $\tau_{\phi_{\psi}}$ is the irreducible representation of $\gll_{n}(F)$ which corresponds to $\phi_{\psi}$ by LLC (local Langlands correspondence) for general linear groups.
\item When $E$ is a field, there is a unique Whittaker datum of $\bu(n)$ if $n$ is odd, and there are exactly two Whittaker data of $\bu(n)$ if $n$ is even.
Thus $J_\psi$ is canonical if $n$ is odd, and there are at most two $J_\psi$ if $n$ is even.
\end{itemize}
We extend the definition of $\Pi_{\psi}$ to elements of $\Psi^{+}(\bu(n))$. 
Let $\psi \in \Psi^{+}(\bu(n))$.
When $E$ is a field, we can decompose $\psi$ as
$$\psi=\xi_1|\cdot|^{r_1}_{E}\oplus\xi_2|\cdot|^{r_2}_{E}\oplus\cdots\oplus\xi_k|\cdot|^{r_k}_{E}\oplus\psi_0\oplus\xi_k^*|\cdot|^{-r_k}_{E}\oplus\cdots \oplus\xi_1^*|\cdot|^{-r_1}_{E}$$
with 
\begin{itemize}
\item real numbers $r_1,\dots, r_k$ such that $r_1\geq r_2\geq\cdots\geq r_k>0$,
\item integers $l_1,\dots,l_k$, $n_0$ such that $n=n_0+2\sum_i l_i$,
\item $\psi_0\in\Psi(\bu(n_0))$, and
\item $l_i$-dimensional irreducible representations $\xi_1,\dots,\xi_k$ such that $\xi_i(W_{E})$ are bounded.
\end{itemize}
Then, we define a multi-set $\Pi_{\psi}$ of representations (there is possibility that they are nonunitary or reducible) by
\begin{align*}\Pi_{\psi}=\{\tau_{\phi_{\xi_1}}|\cdot|^{r_1}_E\times\tau_{\phi_{\xi_2}}|\cdot|^{r_2}_E\times\cdots\times\tau_{\phi_{\xi_k}}|\cdot|^{r_k}_E\rtimes\pi_0 \ | \ \pi_0\in \Pi_{\psi_0} \}\label{eq:a},\end{align*}
where $\tau_{\phi_{\xi_i}}$ is the irreducible representation of $\gll_{l_i}(E)$ which corresponds to $\phi_{\xi_i}$ by LLC for general linear groups.
When $E=F\times F$, we can decompose $\psi$ as
$$\psi=\xi_1|\cdot|^{r_1}_{F}\oplus\xi_2|\cdot|^{r_2}_{F}\oplus\cdots\oplus\xi_k|\cdot|^{r_k}_{F}$$
with 
\begin{itemize}
\item real numbers $r_1,\dots,r_k$ such that $r_1\geq r_2\geq\cdots\geq r_k$
\item integers $l_1,\dots,l_k$ such that $n=\sum_i l_i$, and
\item $l_i$-dimensional irreducible representations $\xi_1,\dots,\xi_k$ which satisfy that $\xi_i(W_{F})$ is bounded.
\end{itemize}
Then, we define $\Pi_{\psi}$ by
\begin{align*}\Pi_{\psi}=\{ \ind_{P(F)}^{\gll_n(F)}\tau_{\phi_{\xi_1}}|\cdot|^{r_1}_F\boxtimes\tau_{\phi_{\xi_2}}|\cdot|^{r_2}_F\boxtimes\cdots\boxtimes\tau_{\phi_{\xi_k}}|\cdot|^{r_k}_F\}\end{align*}
for some parabolic subgroup $P$ of $\bu(n)=\gll_n$.
Since $\mathcal  S_{\psi}\simeq \mathcal S_{\psi_0}$ naturally in both cases, $J_{\psi}$ is defined in an obvious way.
We remark that if $\phi\in \Phi(\bu(n))$, then each element of $\Pi_{\phi}$ has a unique irreducible quotient.
Moreover, the set $\Pi'_\phi$ of the irreducible quotients of elements of $\Pi_\phi$ is equal to the L-packet which corresponds to $\phi$.

For $\eta\in  \hat{\overline {\mathcal  S}}_{\psi}$, we denote by $\pi(\psi,\eta)$ the direct sum of all elements of $\Pi_\psi$ which are mapped to $\eta$, in all cases.

\subsubsection{Properties of the local endoscopic classification}\label{endo-prop}
We highlight some properties of the local endoscopic classification. Let $E$ be a field. 

\begin{enumerate}
\renewcommand{\labelenumi}{(\alph{enumi})}
\item Let $\phi\in\Phi_{\rm bdd}(\bu(n))$. Then $J_\phi$ is injective and each element of $\Pi_{\phi}$ is an irreducible tempered representations of $\bu(n)(F)$.
Moreover, $J_\phi$ is bijective if $F$ is nonarchimedean.
\item When we denote by ${\rm Irr}(\bu(n))$ (resp. ${\rm Irr_{temp}}(\bu(n))$ the set of equivalence class of irreducible representations (resp. the set of irreducible tempered representations) of $\bu(n)(F)$, then
$${\rm Irr}(\bu(n))=\bigsqcup_{\phi\in\Phi(\bu(n))}\Pi'_{\phi}$$and
$${\rm Irr_{temp}}(\bu(n))=\bigsqcup_{\phi\in\Phi_{\rm bdd}(\bu(n))}\Pi_{\phi}.$$
Moreover, when $n=1$, the bijection between ${\rm Irr}(\bu(1))$ and $\Phi(\bu(1))$ is the following map
\begin{align*}
{\rm Irr}(\bu(1))&\leftrightarrow\Phi(\bu(1)),\\
\gamma&\mapsto\check\gamma,
\end{align*}
where $\check\gamma$ is defined by $\check\gamma(\alpha)=\gamma(\alpha/\alpha^c)$ for $\alpha\in E\baz\simeq W_E^{\rm ab}.$
\item Let $\psi\in\Psi^+(\bu(n))$. Then, the all elements of $\Pi_{\psi}$ have the same central character $\omega_\psi$, which corresponds to $\det\circ\psi$.
\item (\cite[Proposition 8.4.1]{mok2015endoscopic}) Let $\psi\in\Psi(\bu(n))$.
Then, $\Pi'_{\phi_\psi}$ is a subset of the set of multiplicity-free elements of $\Pi_\psi$.
Moreover, the following diagram
 $$
  \xymatrix{
    \Pi_{\phi_\psi} \ar[r]^{J_{\phi_\psi}} \ar@{^{(}->}[d] & \hat{ \overline{\mathcal  S}}_{\phi_\psi} \ar@{^{(}->}[d] \\
    \Pi_{\psi} \ar[r]^{J_{\psi}} & \hat{ \overline{\mathcal  S}}_{\psi}
  }
$$
commutes, where $\hat{ \overline{\mathcal  S}}_{\phi_\psi}\map\hat{ \overline{\mathcal  S}}_{\psi}$ is the map induced by the inclusion $S_{\psi}\map S_{\phi_\psi}$ and $\Pi_{\phi_\psi}\map \Pi_{\psi}$ is the inclusion map.
Especially, $\pi(\psi,1)\neq0$.
\item (\cite[Lemma 8.2.2]{mok2015endoscopic}) Let $F$ be a nonarchimedean field.
Let $\psi\in\Psi(\bu(n))$ be trivial on $\{1_{W_F}\}\times\sll_2(\cc)\times\{1_{\sll_2(\cc)}\}$.
For this $\psi$, we define $\hat\phi\in \Phi_{\rm bdd}(\bu(n))$ by $\hat\phi(w,g)=\psi(w,1,g)$ for $(w,g)\in L_{E}$.
Then, the Aubert involution induces a bijection between $\Pi_{\hat\phi}$ and $\Pi_{\psi}$ where the following diagram
   $$
  \xymatrix{
    \Pi_{\hat\phi} \ar[r]^{J_{\hat\phi}} \ar@{<->}[d] & \hat{\overline {\mathcal  S}}_{\hat\phi} \ar@{<->}[d] \\
    \Pi_{\psi} \ar[r]^{J_{\psi}} & \hat{\overline {\mathcal  S}}_{\psi},
  }
$$
commutes, where $\hat{\overline {\mathcal  S}}_{\hat\phi}\leftrightarrow\hat{\overline {\mathcal  S}}_{\psi}$ is the natural isomorphism.
Especially, $J_\psi$ is bijective. 
\item (\cite[Theorem 2.5.1]{mok2015endoscopic}) Let $F$ be a nonarchimedean field and let $\bu(n)$ be unramified.
We denote the inertia group of $W_{E}$ by $I_{E}$.
Then, for $\psi\in\Psi(\bu(n))$ which is trivial on $I_{E}\times\sll_2(\cc)\times\{1_{\sll_2(\cc)}\}$, the representation $J_{\psi}^{-1}(1)$ in $\Pi_{\phi_\psi}$ is the unique unramified representation of $\Pi_\psi$.
\item (See \cite{atobe2017uniqueness}, for example.) Let $\phi\in\Phi(\bu(n))$. If $ \pi\in\Pi'_{\phi}$ is $\mathfrak w$-generic i.e. ${\rm Hom}_{N}(\pi,\lambda)\neq0$, then $\pi$ corresponds to the trivial character.
Moreover, if $\phi\in\Phi_{\rm bdd}(\bu(n))$ then $\Pi_{\phi}$ contains the unique $\mathfrak w$-generic element.
\end{enumerate}
\subsection{Global theory}
Secondly we make a survey of the global endoscopic classification. Let $F$ be a number field.
\subsubsection{Global parameters}
We think about a formal commutative sum as follows:
$$\psi=\boxplus_{i=1}^l k_i\mu_{i}\boxtimes[m_i].$$
Here, 
\begin{itemize}
\item $k_i,n_i,m_i\in\zz_{>0}$ are positive integers which satisfy $\sum_i k_in_im_i=n$,
\item $\mu_{i}$ is an irreducible unitary cuspidal representation of $\gll_{n_i}(\aae)$,
\item $[m_i]$ is the irreducible $m_i$-dimensional algebraic representation of $\sll_2(\cc)$, and
\item $\mu_{i}\boxtimes[m_i]$ (we often suppress $\boxtimes$ and write $\mu_{i}[m_i]$ for short) is a formal tensor product of $\mu_{i}$ of $\gll_{n_i}(\aae)$ and $[m_i]$ which satisfy that if $\mu_i=\mu_j$ and $m_i=m_j$, then $i=j$.
\end{itemize}
We say that $\psi$ is a global discrete A-parameter of $\bu(n)$ if 
\begin{itemize}
\item $k_i=1$ for all $i$ (discreteness),
\item the Asai L-function $\mathrm L(s, \mu_{i}, \mathrm {As}^{(-1)^{m_i+n}})$ (see \cite[\S 7]{gan2012symplectic}) has a simple pole at $s=1$.
\end{itemize}
We denote by $\Psi_2(\bu(n))$ the set of global discrete A-parameters of $\bu(n)$.

For $\psi\in\Psi_2(\bu(n))$, we formally define $S_\psi=O(1,\cc)^l\simeq(\zz/2\zz)^{l}$.
Furthermore, we put $\mathcal S_{\psi}\simeq(\zz/2\zz)^{l}$ and $\mathcal{\overline S}_{\psi}\simeq(\zz/2\zz)^{l-1}$ similar to the local case.
\subsubsection{Localization}\label{loc}
For above $\psi$, we can define the localization $\psi_v$ of $\psi$ at a place $v$ of $F$.
Namely, if $\psi=\boxplus_{i=1}^l\mu_{i}\boxtimes[m_i]$ then 
$$\psi_v=\oplus_{i=1}^l\mu_{i,v}\boxtimes[m_i],$$
where $\mu_{i,v}$ are representations of $L_{E_v}$ correspond to the $v$-th component of $\mu_{i}$ by LLC for general linear groups.
If $\psi\in\Psi_2(\bu(n))$, then $\psi_v\in \Psi^{+}(\bu(n)_v)$.
Moreover, there is a natural homomorphism
$$S_\psi\map S_{\psi_v}$$
and it induces homomorphisms
$${\mathcal S}_\psi\map{\mathcal S}_{\psi_v}
$$
and
$$
\overline{\mathcal S}_\psi\map\overline{\mathcal S}_{\psi_v}
.$$
\subsubsection{The global endoscopic classification}\label{mf}
We fix a global Whittaker datum  $\mathfrak w=(B,\lambda)$ of $\bu(n)$.
Here, $B$ is a Borel subgroup of $\bu(n)$ and $\lambda=\otimes_v\lambda_v$ is a nondegenerate character of $N(\aaf)$ which is trivial on $N(F)$, where $N$ is the unipotent radical of $B$.
We denote by $\mathfrak w_v=(B_{v}, \lambda_v)$ the localization of $\mathfrak w$ at $v$.
For $\psi\in\Psi_2(\bu(n))$, we define a multiset $\Pi_\psi=\otimes_v'\Pi_{\psi_v}$ by
$$\Pi_\psi=\otimes_v'\Pi_{\psi_v}:=\{\otimes_v'\pi_v\ | \pi_v\in\Pi_{\psi_v}, J_{\psi_v}(\pi_v)=1 \ \mbox{for almost all } v\}$$
and for $\eta\in \hat{\overline {\mathcal  S}}_{\psi}=\{\eta\in\hat{{\mathcal  S}}_{\psi} \ | \ \eta(-1,\cdots,-1)=1 \}$, we define
$$\Pi_\psi(\eta)=\{\otimes_v'\pi_v\in\Pi_\psi  \ | \otimes_v J_{\psi_v}(\pi_v)_{\overline{\mathcal S}_\psi}=\eta 
\},$$
where $J_{\psi_v}=J_{\psi_v,\mathfrak w_v}$ and $J_{\psi_v}(\pi_v)_{\overline{\mathcal S}_\psi}$ is defined as the composition of $\overline{\mathcal S}_\psi\map\overline{\mathcal S}_{\psi_v}$ and $J_{\psi_v}(\pi_v)$.
Then, the endoscopic classification claims that the discrete spectrum $L^2_{\rm disc}(\bu(n)(F)\backslash \bu(n)(\aaf))$ of $L^2(\bu(n)(F)\backslash \bu(n)(\aaf))$ is decomposed as
$$L^2_{\rm disc}(\bu(n)(F)\backslash \bu(n)(\aaf))=\bigoplus_{\psi \in \Psi_2(\bu(n))}L^2_{\rm \psi},$$
where
$$L^2_{\rm \psi}=\bigoplus_{\pi\in\Pi_\psi(\epsilon_\psi)}\pi$$
for each $\psi\in \Psi_2(\bu(n))$ and $\epsilon_\psi$ is the character of $\overline{\mathcal S}_{\psi}$ which is defined by a root number (see \cite[p.47]{arthur2011endoscopic} or \cite[p.34]{mok2015endoscopic}).

By (e) in \S \ref{endo-prop}, there is a natural injection from the set of near equivalence classes of irreducible discrete automorphic representations of $\bu(n)(\aaf)$ to $\Psi_2(\bu(n))$.
\begin{rmk}
Today the above results are generalized to nonquasisplit unitary groups by Kaletha, Minguez, Shin and White \cite{kltarticle}. 
\end{rmk}

\section{Ikeda lifts and Miyawaki lifts}\label{S:im}
The Ikeda lift and the Miyawaki lift are kinds of liftings both established by Ikeda.
The origin of the Ikeda lift is \cite{ikeike}, in which Ikeda gave a construction of Siegel modular forms from elliptic modular forms.
Furthermore, he gave Hermitian analogue in \cite{ikeda2008lifting}.
After that, he gave the representation-theoretical Ikeda lift in \cite{ikeda2017lifting} with Yamana.
In addition, Kim and Yamauchi gave the $E_7$ analogue in \cite{kim2016cusp}.

The origin of the Miyawaki lift is \cite{ikeda2006pullback}, in which Ikeda gave a construction of Siegel modular from another Siegel modular forms by using diagonal restrictions of Ikeda lifts as kernel functions.
There are also some analogues of the result for other groups: the unitary group analogue by Atobe and Kojima \cite{atobe2018miyawaki}, and  the GSpin$(2,10)$ analogue by Kim and Yamauchi \cite{kim2018miyawaki}.
Recently, Atobe constructed and studied the representation-theoretical Miyawaki lift for symplectic/metaplectic groups \cite{attarticle}.

In this section, we redefine the Ikeda lift and the Miyawaki lift for unitary groups in terms of the endoscopic classification. 

For convenience, we put 
$$J'_{2n}=\begin{pmatrix}
 &J_n \\
 -J_n&
\end{pmatrix}$$
 and let
\begin{align*}
\mathrm {GU}(n,n)&=\{g \in \res\gll_{2n} \ | \ {}^t\! g^cJ'_{2n}g=\nu_{2n}(g)J'_{2n}, \nu_{2n}(g)\in \gll_1 \},\\
\mathrm {U}(n,n)&=\mathrm {ker}(\nu_{2n}).
\end{align*}
We identify ${\rm GU}(2n)$ with ${\rm GU}(n,n)$ by the following isomorphism
\begin{align*}
\mathrm {GU}(2n)&\bij\mathrm {GU}(n,n)\\
g&\mapsto (\begin{smallmatrix}
 1_n&\\
 &\kappa 1_n
\end{smallmatrix})
g
(\begin{smallmatrix}
 1_n&\\
 &\kappa 1_n
\end{smallmatrix})^{-1}.
\end{align*}
We identify ${\rm U}(2n)$ with ${\rm U}(n,n)$ similarly.

In this section, we assume that $F$ is always a number field.
We denote by $\ep=\otimes_v\epsilon_{E_v/F_v}$ the character of $\aaf^\times/F^\times$ which corresponds to the extension $E/F$ by the class field theory.
\subsection{Ikeda lifts}

We fix $n\in\zz_{\geq1}$.
Let $\pi$ be an irreducible unitary cuspidal automorphic representation of ${\gll_{2}}(\aaf)$ with central character $\ep^{n-1}$.
We denote by ${\phi}={\phi_{\pi}}$ the (weak) base change lift of $\pi$ to $\gll_2(\aae)$.
Since ${\phi}$ has trivial central character, so it is selfdual.
We fix an automorphic character $\chi$ of $\aae^{\times}$ such that $\chi|_{\aaf^{\times}}=\ep^{n-1}$.
Then we can regard the representation $\pi\boxtimes\chi$ of $\gll_2(\aaf)\times\aae\baz$ as an irreducible cuspidal automorphic representation of $\guu(\aaf)$ via the following map
\begin{align*}(\gll_2\times \res \gll_1)/\{(a {\bf 1}_2, a^{-1})\ |\ a\in\gll_1\}&\simeq\guu ,\\
(g,z)&\map gz.
\end{align*}
Then, the following map
 \begin{align*}
 \pi\boxtimes\chi &\map L^2(\uu(F)\backslash\uu(\aaf)) \\
 f &\mapsto f|_{\uu(\aaf)}
\end{align*}
is obviously nonzero.
Moreover, it is easy to verify that all irreducible subrepresentation of the image of this map are nearly equivalent and the standard base change of them is equal to the  automorphic representation 
$$\phi_{\chi}=\phi_{\pi,\chi}:=\phi\otimes(\chi\circ\det)$$ 
of $\gll_2(\aae)$.
Especially, $\phi_{\chi}$ is a discrete A-parameter of $\uu$.
Thus we obtain that
$$\psi=\psi_{\pi,n}:={\phi}[n]$$
 is a discrete A-parameter of $\bu(n,n)$.

Since we can easily to verify $S_{\phi_\chi} \simeq S_\psi$, $S_{(\phi_\chi)_v} \simeq S_{\psi_v}$ for any $v$ and $\epsilon_{\phi_\chi}=\epsilon_\psi=1$, we can redefine Ikeda lifts as follows.
\begin{defi}
For an irreducible subrepresentation $\tau=\otimes_v\pi((\phi_\chi)_v,\eta_v)$ of $L^2_{\phi_\chi}$, we  put
$$I_{n}(\tau)=I_{n}(\tau,\chi):=\otimes_v\pi(\psi,\eta_v).$$
We call this automorphic representation of $\bu(n,n)(\aaf)$ the Ikeda lift of $\tau$ (with respect to the fixed Whittaker datum).
\end{defi}

\subsection{Types of $\phi$ and $\phi_v$}\label{im:type}
For convenience, we classify the parameter $\phi$. In the following table, $p(\mu)=b$ means that $\mu$ is conjugate selfdual with parity $b$.
\renewcommand{\arraystretch}{1.3}
\begin{table}[H]
  \begin{tabular}{ccccc} \hline
    type & $\phi$ & $S_{\psi}$ & $\overline{\mathcal S}_{\psi}$ & details \\ \hline 
    1 & $\mu_2$ &  $O(1,\cc)$ &  $\{1\}$ & $\mu_2$:irr. cusp. repn. of $\gll_2(\aae)$, $\mu_2=\mu_2^{\vee}, p(\mu_2)=(-1)^n$ \\ 
    2 & $\mu_1\boxplus\mu_1^{-1}$ & $O(1,\cc)^{2}$ & $\zz/2\zz$ & $\mu_1$:irr. cusp. repn. of $\gll_1(\aae)$, $\mu_1\neq\mu_1^{\vee}, p(\mu_1)=(-1)^n$ \\ \hline
  \end{tabular}
\end{table}
Let $v$ be a place of $F$ which does not split in $E$. Let us classify the parameter $\phi_v$ similarly to the global one. We also define $p$ similarly to the global one.
\begin{table}[H]
  \begin{tabular}{ccccc} \hline
    type     & $\phi_v$        & $S_{{\psi}_v}$         & $\overline{\mathcal S}_{{\psi}_v}$    & details \\ \hline
    1-1 & $\mu_2$                &  $O(1,\cc)$   &  $\{1\}$       & $\mu_2$:irr. and $2$-dim., $\mu_2=\mu_2^{\vee}, p(\mu_2)=(-1)^n$ \\
    1-2 & $\mu_1[2]$             & $O(1,\cc)$    & $\{1\}$        & $\mu_1$:$1$-dim., $\mu_1=\mu_1^{\vee}, p(\mu_1)=(-1)^{n-1}$ \\ 
    2-1 & $\mu_1\oplus\mu_1^{-1}$&  $O(1,\cc)^2$   & $\zz/2\zz$    & $\mu_1$:$1$-dim., $\mu_1\neq\mu_1^{\vee}, p(\mu_1)=(-1)^n$ \\ 
    2-2 & $2\mu_1$               & $O(2,\cc)$    & $\zz/2\zz$        & $\mu_1$:$1$-dim., $\mu_1=\mu_1^{\vee}, p(\mu_1)=(-1)^n$ \\ 
    3   & $\xi\oplus \xi^{-1}$      &  $\sll_2(\cc)$ or $\gll_1(\cc)$   &  $\{1\}$     & $\xi$:$1$-dim., $\xi=\xi^c$, $[\xi=\xi^{*}\Rightarrow p(\xi)=(-1)^{n-1}]$ \\ \hline
  \end{tabular}
\end{table}
\renewcommand{\arraystretch}{1}
Let us make a few remarks about them:
\begin{itemize}
\item If $\phi$ is of type 2, then $\phi_v$ can only be of type 2-1 or 2-2 for any $v$.
\item For each archimedean place $v$ of $F$, $\phi_v$ can only be of type 2-1, 2-2 or 3.
\end{itemize}
\subsection{Some properties of A-packet $\Pi_{\psi_v}$}
We show some basic properties of $\Pi_{\psi_v}$.
\begin{lemm}\label{L:fr}
Let $v$ be nonarchimedean.
 Then, each element of $\Pi_{\psi_v}$ is irreducible, $\Pi_{\psi_v}$ is multiplicity free, and $J_{\psi_v}$ is  bijective.
\end{lemm}
\begin{proof}
If $v$ splits in $E$, then the unique element of $\Pi_{\psi_v}$ is irreducible because of Ramanujan bound.
Therefore we can assume that $v$ does not split in $E$.
Then this lemma follows from (e) in \S\ref{endo-prop} and Ramanujan bound except that $\phi_v$ is of type 1-2, so we assume that $\phi_v=\mu_1[2][n]$ for some $\mu_1$.
We only need to show is the irreducibility of the unique element $\pi(\psi_v,1)$ of $\Pi_{\psi_v}$.

If $n=1$ then that is follows from (a) in \S\ref{endo-prop}, so we assume $n>1$.
Then, by the similar discussion in Proposition 5.9 of \cite{xu_2017}, we have
$$\pi(\psi_v,1)=|\cdot|_{E_v}^{\frac{1}{2}}\langle \ind^{\gll_n(E)}_B\mu_1|\cdot|_{E_v}^{-\frac{n}{2}}\boxtimes\cdots \boxtimes \mu_1|\cdot|_{E_v}^{\frac{n}{2}}\rangle\rtimes 1_{\bu(1)}\ominus\left(\bigoplus_{\pi'\in\Pi_{\mu_1[1][n+1]\oplus\mu_1[1][n-1]}}\pi'\right)$$
as an element of the Grothendieck group of representations of $\bu(n,n)$, where $B$ is the Borel subgroup of $\gll_n(E)$ which consists of upper triangular elements and $\langle \ind^{\gll_n(E)}_B\mu_1|\cdot|_{E_v}^{-\frac{n}{2}}\boxtimes\cdots \boxtimes \mu_1|\cdot|_{E_v}^{\frac{n}{2}}\rangle$ is the maximal semisimple subrepresentation of $\ind^{\gll_n(E)}_B\mu_1|\cdot|_{E_v}^{-\frac{n}{2}}\boxtimes\cdots \boxtimes \mu_1|\cdot|_{E_v}^{\frac{n}{2}}$.
It is easy to verify that 
$$\langle \ind^{\gll_n(E)}_B\mu_1|\cdot|_{E_v}^{-\frac{n}{2}}\boxtimes\cdots \boxtimes \mu_1|\cdot|_{E_v}^{\frac{n}{2}}\rangle=\mu_1\circ\det,$$
 so the length of $|\cdot|_{E_v}^{\frac{1}{2}}\langle \ind^{\gll_n(E)}_B\mu_1|\cdot|_{E_v}^{-\frac{n}{2}}\boxtimes\cdots \boxtimes \mu_1|\cdot|_{E_v}^{\frac{n}{2}}\rangle\rtimes 1_{\bu(1)}$ is $3$ by \cite[Theorem 1.2]{ksdeg}.
On the other hand, $\#\Pi_{\mu_1[1][n+1]\oplus\mu_1[1][n-1]}=2$ by (e) in \S\ref{endo-prop}.
Thus $\pi(\psi_v,1)$ is irreducible.
\end{proof}
For archimedean places, the following holds by \cite[Th\'eor\`eme 1.3]{MR2019}.
\begin{lemm}
If $v$ is archimedean, then $\Pi_{\psi_v}$ is multiplicity free.
\end{lemm}
We note that above lemmas lead the following.
\begin{cor}\label{c:mf}
$L_\psi^2$ is multiplicity free.
\end{cor}
\begin{rem}
When $v$ is archimedean, we can determine $\Pi_{\psi_v}$ in terms of cohomological induction (\cite{MR2019}, see \S\ref{det}).
In contrast to the nonarchimedean case, the elements of $\Pi_{\psi_v}$ are not always irreducible.
\end{rem}
\begin{lemm}\label{L:glza}
Let $\pi=\pi(\psi_v,\eta)\neq0$.
Then $\pi$ has a globalization i.e. a nonzero subrepresentation $\dot\pi=\otimes_{v'}\pi(\psi_{v'},\eta_{v'})$ of $L_{\psi}^2$ such that $\eta=\eta_{v}$.
\end{lemm}
\begin{proof}
If characters $\{\eta_{v'}\}_{v'\neq v}$ of $\overline{\mathcal S}_{\psi_v'}\ (v'\neq v)$ satisfy
$$(\otimes_{v'\neq v}\eta_{v'\overline{\mathcal S}_\psi})\otimes\eta_{\overline{\mathcal S}_\psi}=1$$
and $\eta_{v'}=1$ for all archimedean place and almost all nonarchimedean place $v'\neq v$, then the representation $\otimes_{v'\neq v}\pi(\psi_{v'},\eta_{v'})\otimes\pi(\psi_{v},\eta)$
is a nonzero (by Lemma \ref{L:fr} and (d) in \S\ref{endo-prop}) subrepresentation of $L_\psi$,
so we need to do is to find such $\{\eta_{v'}\}_{v'\neq v}$.
If $\phi$ is of type 1 or $\eta=1$, then we can take $\eta_{v'}=1$ for all $v'\neq v$.
Therefore we assume that $\phi$ is of type 2 and $\eta\neq1$ (necessarily $v$ does not split in $E$).
Then $\overline{\mathcal S}_\psi\simeq \overline{\mathcal S}_{\psi_v}\simeq\zz / 2\zz$ and $\eta=-1$ as a character of $\zz / 2\zz$.
We take a finite place $v''\neq v$ of $F$  which does not split in $E$.
Since $\overline{\mathcal S}_{\psi_{v''}}\simeq\zz / 2\zz$, we can take $\eta_{v'}=1$ for all $v'\neq v,v''$ and $\eta_{v''}=-1$.
\end{proof}

\begin{lemm}
If $n$ is even, then the labeling $J_{\psi_v}$ of $\Pi_{\psi_v}$ does not depend on Whittaker data.
\end{lemm}
\begin{proof}
The lemma is obvious if $\#\Pi_{\psi_v}= 1$, so we assume that $\psi_v$ is of type 2-1 or 2-2. 
Then, $\pi(\psi_v,1)$ is irreducible.
It follows from Lemma \ref{L:fr} if $v$ is nonarchimedean.
For archimedean places, we discuss in \S \ref{det}.
Since $\#\Pi'_{\phi_{\psi_v}}=1$, $\pi(\psi_v,1)$ is equal to the unique element of $\Pi'_{\phi_{\psi_v}}$ by (d) in \S\ref{endo-prop}.
Especially, $\pi(\psi_v,1)$ does not depend on Whittaker data.
Since $\overline{\mathcal S}_{\psi_v}\simeq \zz/2\zz$, $\pi(\psi_v,-1)$ also does not depend on Whittaker data.
Therefore this lemma holds.
\end{proof}
\begin{lemm}\label{L:fix}
Let $n$ be even. Then, for any nonarchimedean place $v$, each element of $\Pi_{\psi_v}$ is fixed by the actions of ${\rm GU}(n,n)(F_v)$ induced by the conjugation.
\end{lemm}
\begin{proof}
If $v$ splits in $E$, then the action of ${\rm GU}(n,n)(F_v)$ on the set of representations of ${\rm U}(n,n)(F_v)$ is trivial because ${\rm GU}(n,n)(F_v)$ is generated by ${\rm U}(n,n)(F_v)$ and the center of ${\rm GU}(n,n)(F_v)$.
Thus we can suppose that $v$ does not split in $E$. 
We fix $a\in F^\times \setminus N_{E_v/F_v}(E_v^\times)$. 
Then, $\{1_{2n}, \varepsilon_a= \left(\begin{smallmatrix}1_n&0\\ 0&a1_n \end{smallmatrix}\right)\}$ is a complete system of representatives of 
$${\rm GU}(n,n)(F_v)/E_v^\times{\rm U}(n,n)(F_v)\simeq\{\pm 1\}.$$
Firstly $\pi(\psi_v,1)$ is the unique irreducible quotient of the following standard module
$$\pi'=\tau_{\phi_{v}}|\det|^{\frac{n-1}{2}}\times\tau_{\phi_{v}}|\det|^{\frac{n-3}{2}}\cdots\times\tau_{\phi_{v}}|\det|^{\frac{1}{2}}\rtimes 1_{\bu(0)}$$
by Lemma \ref{L:fr}, where $\tau_{\phi_{v}}$ is the irreducible representation of $\gll_2(E_v)$ which corresponds to $\phi_v$.
Since $\varepsilon_a$ acts trivially on the Levi factor ($\simeq\gll_2(E_v)^{\frac{n}{2}}$), $\pi'$ is equivalent to $\pi'^{\varepsilon_a}$.
Thus $\pi(\psi_v,1)\simeq\pi(\psi_v,1)^{\varepsilon_a}$ holds.

Secondly we take $\pi\in\Pi_{\psi_v}$ in general. We take a globalization $\dot{\pi}$ by Lemma \ref{L:glza}.
Since $\dot{\pi}^{\varepsilon_a}$ is an automorphic representation of ${\rm U}(n,n)(\aaf)$ which satisfies $\dot{\pi}^{\varepsilon_a}_{v'}\simeq\dot{\pi}_{v'}$ for almost all places $v'$, $\dot{\pi}^{\varepsilon_a}$ is also a subrepresentation of $L^2_{\psi}$.
Thus we have $\pi^{\varepsilon_a}\in\Pi_{\psi_v}$.
Since $\#\Pi_{\psi_v}\leq 2$, we conclude $\pi^{\varepsilon_a}\simeq\pi$.
\end{proof}
\begin{rem}
Lemma \ref{L:fix} does not hold if $n$ is odd.
The fact is maybe one reason that Yamana's construction of Ikeda lifts \cite{yayayaya} is not for ${\rm U}(n,n)$ but ${\rm GU}(n,n)$. 
\end{rem}
\subsection{Miyawaki lifts}\label{miyaa}
Let $n,\pi,\phi,\chi,\phi_{\chi},\psi$ be as above.
We take a positive integer $m\leq n$.
Let $V_1$ be a nondegenerate $m$-dimensional Hermitian space over $E$ and let $V_2$ a nondegenerate $(2n-m)$-dimensional Hermitian space over $E$.
We put $V=V_1\perp V_2$.
Then, there is a natural embedding
$$i=i_{V_1,V_2}:\bu(V_1)\times\bu(V_2)\inj\bu(V).$$

For simplicity, we assume that $\bu(V)$, $\bu(V_1)$, and $\bu(V_2)$ are quasisplit i.e. the dimension of anisotropic space of $V_1$ is at most one and $V_2=V_1^-\perp\mathbb  H_E^{n-m}$, where $\mathbb H_E$ is the nondegenerate $2$-dimensional isotropic Hermitian space and $V_1^-$ is the Hermitian space such that $V_1\perp V_1^-=\mathbb H_E^{m}$.
Let $\psi'$ be a discrete A-parameter of $\bu(V_1)$.
Then, we define a parameter $\mathcal M_{\psi}(\psi')$ by
$$\mathcal M_{\psi}(\psi')={\phi}[n-m]\boxplus\psi'^\vee.$$
Here, if $m=n$ then we regard $\phi[0]$ as the zero of $\boxplus$.
We note that $\mathcal M_{\psi}(\psi')$ may not be discrete A-parameter of $\bu(V_2)$.
For (not necessarily irreducible) representations $\Pi\subset L^2_{\psi}$ and $\pi\subset L^2_{\psi'}$, we define $\mathcal M_{\Pi}(\pi)$ as the representation of $\bu(V_2)(\aaf)$ generated by automorphic forms $\mathcal M_{\mathcal F}(f)$ defined by
\begin{align*}\mathcal M_{\mathcal F}(f)(g)=\int_{\bu(V_1)(F)\backslash\bu(V_1)(\aaf)}\mathcal F\circ i(g,h)\overline{f(h)}dh  \ \ (\mathcal F\in\Pi, f\in \pi)\end{align*}
as long as all of them converge.
Then, the following holds.
\begin{prop}\label{P:ee}
We assume that $\mathcal M_{\Pi}(\pi)\subset L^2(\bu(V_2)(F)\backslash\bu(V_2)(\aaf))$ and $\mathcal M_{\psi}(\psi')\in{\Psi_2(\bu(V_2))}$. Then,
$$\mathcal M_{\Pi}(\pi)\subset L^2_{\mathcal M_{\psi}(\psi')}.$$
\end{prop}
\begin{proof}
By  \cite[\S 2]{atobe2018miyawaki} (the proofs of the propositions also works when $m$ is odd), the near equivalence class of $\mathcal M_{\Pi}(\pi)$ corresponds to $\mathcal M_{\psi}(\psi')$. By the similar discussion in \cite{arem}, $\mathcal M_{\Pi}(\pi)$ is contained in the discrete spectrum $L_{\rm disc}(\bu(V_2)(F)\backslash\bu(V_2)(\aaf))$, so $\mathcal M_{\Pi}(\pi)\subset L^2_{\mathcal M_{\psi}(\psi')}$.
\end{proof}
We call $\mathcal M_{\Pi}(\pi)$ the Miyawaki lift of $\Pi$ with respect to $\pi$. The definition is a representation-theoretical redefinition of \cite{atobe2018miyawaki} and the unitary group analogue of \cite{attarticle}.
\begin{rmk}
The above proposition is only for quasisplit groups, but it is easy to generalize it for nonquasisplit cases because we only need the information of near equivalence classes of automorphic representations to prove it.
\end{rmk}

\subsection{Main theorem}\label{ss:mt}
In this subsection, we state our main result.
It determines Miyawaki lifts for quasisplit unitary groups in the case that $n=2$ and $m=1$ explicitly.

Let $n=2$ and let $\pi,\phi,\chi,\phi_{\chi},\psi$ as above.
We fix a global Whittaker datum $\mathfrak w=(B_2,\lambda_{\psi_F})$ of $\bu(2)$, where $\lambda_{\psi_F}$ is defined by
$$\lambda_{\psi_F}\left(\left(\begin{smallmatrix}1&\kappa\alpha\\&1\end{smallmatrix}\right)\right)=\psi_F(\alpha),\  \alpha\in\aaf$$
(we note that $\left(\begin{smallmatrix}1&\kappa\alpha\\&1\end{smallmatrix}\right)$ is identified with $\left(\begin{smallmatrix}1&\alpha\\&1\end{smallmatrix}\right)$ in $\guu(\aaf)$).
To describe our result simply, we assume that $\psi_{F_v}$ is positive with respect to $\kappa\in E\baz\subset E_v\baz$ for each place $v$ of $F$ such that $E_v\simeq\cc$.
Here, we define the term ``positive with respect to $\kappa$'' as follows:
\begin{itemize}
\item Let $F=\rr$ and let $E$ be a quadratic extension of $F$.
Let $\kappa\in E\baz$ be a trace-zero element.
We take $\psi$ to be a nontrivial character of $F$.
Then there is a nonzero real number $b$ such that $\psi(x)=e^{ibx}$ for all $x\in F$, where $i$ is a fixed imaginary unit. 
Moreover, for each isomorphism $\varphi:E_v\simeq \cc$, the number $c=\varphi^{-1}(i)b/\kappa$ is in $\rr\baz$ and it does not depend on $\varphi$ (and $i$).
Following the above fact, we say that $\psi$ is positive if $c$ is positive.
\end{itemize}
Let $\tau=\otimes_v\pi((\phi_\chi)_v,\eta_v)$ be an irreducible (necessarily cuspidal) automorphic representation of $\uu(\aaf)$ with A-parameter $\phi_\chi$ and $\gamma=\otimes_v\gamma_v$ an automorphic character of $\bu(1)(\aaf)$.
We note that the A-parameter of $\gamma$ is $\check\gamma=\otimes_v\check\gamma_v$ (see (b) in \S\ref{endo-prop}).
Then, the main theorem is as follows.
\begin{theo}\label{T:main}
We assume that $\phi\neq\check\gamma\boxplus\check\gamma^{-1}$, namely $\mathcal M_{\psi}(\check \gamma)=\phi\boxplus\check\gamma^{-1}\in\Psi_2(\bu(3))$. Let $\mathbb H_E^2=V_1\perp V_2$ be an orthogonal decomposition of $\mathbb H_E^2$ into a 1-dimensional subspace $V_1$ and a $3$-dimensional subspace $V_2$.
We regard $I_2(\tau)$ as an automorphic representation of $\bu(\mathbb H_E^2)(\aaf)$ and $\gamma$ as an automorphic character of $\bu(V_1)(\aaf)$.
Then, 
$$\mathcal M_{I_2(\tau)}(\gamma)\subset L^2_{\mathcal M_{\psi}(\check \gamma)}$$
as an automorphic representation of $\bu(3)(\aaf)$.
Moreover,
$$\mathcal M_{I_2(\tau)}(\gamma)\simeq\otimes_v \sigma_v,$$
where $\sigma_v$ is given as follows:

\begin{enumerate}
\renewcommand{\labelenumi}{(\alph{enumi})}
\item When $v$ splits in $E$, $\sigma_v$ is the irreducible representation of $\gll_3(F_v)$ which corresponds to $\mathcal M_{\psi}(\check \gamma)_v$.

\item When $v$ is nonarchimedean and $v$ does not split in $E$, then there are two cases.
\begin{enumerate}
\renewcommand{\labelenumii}{(\roman{enumii})}
\item If $\phi_v\neq\check\gamma_v\oplus\check\gamma_v^{-1}$, then 
$$\sigma_v=\pi(\mathcal M_{\psi}(\check\gamma)_v,(\eta_v,{\bf 1}_{\zz/2\zz}))$$
under the identification
$$\mathcal S_{\mathcal M_{\psi}(\check\gamma)_v}=\mathcal S_{(\phi_\chi)_v}\oplus(\zz / 2\zz).$$

\item If $\phi_v=\check\gamma_v\oplus\check\gamma_v^{-1}$, then
$$\sigma_v=\begin{cases}
\pi(\mathcal M_{\psi}(\check\gamma)_v,({\bf 1}_{\mathcal S_{\mathcal M_{\psi}(\check\gamma)_v}})) &\mbox{if }\eta_v={\bf 1}_{\mathcal S_{(\phi_\chi)_v}} ;\\
0 &\mbox{otherwise}.
\end{cases}$$
\end{enumerate}

\item When $v$ is archimedean and does not split in $E$, then there are two cases. 
We define a character $\chi_\alpha$ of $\cc\baz$ for $\alpha\in\zz$ by $\chi_\alpha(r e ^{i\theta})=e^{i\alpha\theta}$ $(r\in\rr_{>0}, \ \theta\in \rr/2\pi \zz)$, and write
$$\phi_v=|\cdot|^t\chi_k\oplus|\cdot|^{-t}\chi_{-k}, \ \check\gamma_v=\chi_{2s},$$
 where $k\in\zz_{\geq 0}$, $t\in\rr$ and $s\in\zz$ .  
\begin{enumerate}
\renewcommand{\labelenumii}{(\roman{enumii})}
\item If $t \neq 0$ or $k$ is odd or $|2s|<k$, then
$$\sigma_v=\pi(\mathcal M_{\psi}(\check\gamma)_v,(\eta_v,{\bf 1}_{\zz/2\zz}))$$
under the identification
$$\mathcal S_{\mathcal M_{\psi}(\check\gamma)_v}=\mathcal S_{(\phi_\chi)_v}\oplus(\zz / 2\zz).$$

\item If $t = 0$ and $k$ is even and $k\leq|2s|$, then
$$\sigma_v=\begin{cases}
\pi(\mathcal M_{\psi}(\check\gamma)_v,({\bf 1}_{\mathcal S_{\mathcal M_{\psi}(\check\gamma)_v}})) &\mbox{if }\eta_v={\bf 1}_{\mathcal S_{(\phi_\chi)_v}} ;\\
0 &\mbox{otherwise}.
\end{cases}$$
Especially, $\mathcal M_{I_2(\tau)}(\gamma)\neq 0$ if and only if $\sigma_v\neq 0$ for each $v$.
Furthermore, $\mathcal M_{I_2(\tau)}(\gamma)$ is irreducible if $\mathcal M_{I_2(\tau)}(\gamma)\neq 0$ .
\end{enumerate}
\end{enumerate}
\end{theo}
\begin{rem}\label{emb}
\begin{itemize}
\item $\mathcal M_{I_2(\tau)}(\gamma)$ does not depend on the choice of $V_1$ (and $V_2$).
We have it by the following two facts immediately:
\begin{enumerate}
\renewcommand{\labelenumi}{(\alph{enumi})}
\item $I_2(\tau)$ is stable under the conjugation of $\guuuu(F)$ (see \S \ref{S:hm}).
\item For any other pair $(V'_1,V'_2)$, we can take $g\in\guuuu(F)$ such that ${{\rm Im}i_{V'_1,V'_2}}=g^{-1}{{\rm Im}i_{V_1,V_2}}g.$ 
\end{enumerate}
We can obtain (b) as follows.
We fix vectors $0\neq v \in V_1$, $0\neq v' \in V'_1$ and put $\langle v,v\rangle=a\in E\baz$
$\langle v',v'\rangle=a'\in E\baz$, where $\langle \cdot,\cdot\rangle$ is the Hermitian form on $\mathbb H^2_E$.
Then, we can take $g_1\in \guuuu(F)$ such that $\langle g_1\cdot,g_1\cdot\rangle=\frac{a}{a'}\langle \cdot,\cdot\rangle$ since $\mathbb H^2_E$ is hyperbolic.
By Witt's theorem, there is an element $g_2$ of $\uuuu(F)$ such that $g_2g_1v'=v$.
The element $g=g_2g_1$ of $\guuuu(F)$ is what we want.
\item We can regard that Atobe's paper \cite{AtoH:2015} treats the nonvanishing problem of Miyawaki lifts for $\bu(2)$ and $\bu(2)$ with respect to Hermitian Maass lifts. 
\end{itemize}
\end{rem}
\section{Hermitian Maass lifts}\label{S:hm}
In the classical setting, the Hermitian Maass lift means the Hermitian analogue of the Saito-Kurokawa lift \cite{kojima1982arithmetic}.
Following this, we call the Ikeda lift of degree 2 the Hermitian Maass lift.
The aim of this section is to describe representation-theoretical Hermitian Maass lifts in terms of theta lifts.
It is well-known in the classical setting (see \cite{AtoH:2015}).

\subsection{Overview}\label{her-ov}
(Some undefined symbols which appear in the following sentences are defined in the bottom of this subsection.)
Let $F$ be a number field.
Let $n=2$ and let $\pi, \phi, \chi, \phi_\chi, \psi$ be as in \S\ref{ss:mt}.
Let $\tau=\otimes_v\pi((\phi_\chi)_v,\eta_v)$ (implicitly we fix a Whittaker datum of $\bu(2)$ as in \S\ref{ss:mt}).
We put $\tilde \tau=(\tau\boxtimes\chi)|_{\gll^+_{2,\aaf}}$, where the representation $\tau\boxtimes\chi$ of $\uu(\aaf)\times\aae\baz$ is regarded as a representation of $\guu^+_{\aaf}$ by the equation $\guu^+_{\aaf}=\aae\baz\uu(\aaf)$. We regard $\tilde \tau$ as a space of functions on $\gll_{2,\aaf}^{+}$  which are left
$\gll_{2,F}^+$-invariant.
Let $\theta^2(\tilde \tau)$ be the theta lift of $\tilde \tau$ to $\go$ (see \S\ref{SS: tl}).
It is easy to verify that $\theta^2(\tilde \tau)$ is nonzero and square integrable modulo center by the results about theta lifts for isometry groups \cite[Proposition 10.1]{yamana2014functions}.
Since $\theta^2(\tilde \tau)$ has the trivial central character (see \S\ref{SS: tl}), we can define $\mathcal H(\tau)=\mathcal H(\tau,\chi)$ as the image of the following map
\begin{align*}
\theta^2(\tilde \tau)&\map L^2(\uuuu(F)\backslash\uuuu(\aaf)),\\
f&\mapsto (f\circ j)|_{\uuuu(\aaf)}.
\end{align*}
Here, $j$ is the isomorphism $j:{\rm PGU}(2,2)\simeq{\rm PGSO}(4,2)$ (see \S\ref{her-isom}) and we regard $f\circ j$ as an automorphic form of $\guuuu(\aaf)$ with the trivial central character.
We note that $\mathcal H(\tau)$ is nonzero because  $\theta^2(\tilde \tau)$ is nonzero.
In this section, we will prove the following:
\begin{theo}\label{T: her}
$\mathcal H(\tau)=I_2(\tau)$.
\end{theo}

We introduce some additional notations.
Let $F$ be a global or local field.
We put $$S_0=
\begin{pmatrix}
 2&0\\
 0&-2d\\
\end{pmatrix}$$
and  
$$S_m=
\begin{pmatrix}
 &&1\\
 &S_{m-1}&\\
 1&&\\
\end{pmatrix}$$
for $m>0$ inductively (we recall $d=\kappa^2$).
We define some algebraic groups over $F$ as follows:
\begin{align*}
\mathrm {GO}(m+2, m)&=
\{g \in \gll_{2m+2} \ | \ {}^t\! gS_mg=\nu'_{m}(g)S_m, \nu'_{m}(g)\in \gll_1 \},\\
\mathrm {O}(m+2, m)&=\mathrm {ker}(\nu'_{m}),\\
\mathrm {GSO}(m+2, m)&=
\{g \in \mathrm {GO}(m+2, m) \ | \det g= \nu'_{m}(g)^{m+1} \},\\
\mathrm {SO}(m+2, m)&=\mathrm {O}(m+2, m) \cap\mathrm {GSO}(m+2, m).
\end{align*}
We often denote $\nu'_{m}$ by $\nu$ for short.
We define 
a subgroup $B^{GO}_m$ of $\mathrm {GO}(m+2, m)$ by
$$B^{GO}_m=\left\{\begin{pmatrix}
 \gll_{1}&*&*&*&*&*&*\\
 &\ddots&*&*&*&*&*\\
 &&\gll_{1}&*&*&*&*\\
 &&&\mathrm {GO}(2,0)&*&*&*\\
 &&&&*&*&*\\
 &&&&&\ddots&*\\
 &&&&&&*
\end{pmatrix}\right\} \subset \mathrm {GO}(m+2, m).
$$
Furthermore, we define 
$B^{O}_m,B^{GSO}_m,B_m^{SO}$ of $\mathrm {O}(m+2, m), \mathrm {GSO}(m+2, m), \mathrm {SO}(m+2, m)$ respectively by
\begin{align*}
B^{O}_m&=B^{GO}_m\cap\mathrm {O}(m+2, m),\\
B^{GSO}_m&=B^{GO}_m\cap \mathrm {GSO}(m+2, m),\\
B^{SO}_m&=B^{GO}_m\cap\mathrm {SO}(m+2, m).
\end{align*}
When $F$ is a local field, we denote by $\tau_1\times\tau_2\times\cdots\times\tau_{k}\rtimes\pi_0$ a normalized parabolic induction
$${\rm Ind}_{P^{GO}}^{\mathrm {GO}(m+2, m)}\tau_1\boxtimes\tau_2\boxtimes\cdots\boxtimes\tau_k\boxtimes\pi_0,$$
where
\begin{itemize}
\item $l_1, \dots ,l_k$ are positive integers such that $m_0=m-\sum_i l_i\geq 0$,
\item $P^{GO}$ is the parabolic subgroup of $\mathrm {GO}(m+2, m)$ containing $B^{GO}_{m}$ whose Levi subgroup is isomorphic to $\prod_i \gll_{l_i}\times \mathrm {GO}(m_0+2,m_0)$,
\item $\tau_1,\dots,\tau_k,$ and $\pi_0$ are representations of $\gll_{l_1},\dots,\gll_{l_k},$ and $\mathrm{GO}(m_0+2,m_0)$, respectively,
\item $\tau_1\boxtimes\tau_2\boxtimes\cdots\boxtimes\tau_k\boxtimes\pi_0$ is regarded as a representation of $P^{GO}$ through the map
\begin{alignat*}{3}
P^{GO} \quad\quad\quad\quad&\surj \prod_i \gll_{l_i}\times \mathrm {GO}(m_0+2,m_0),\\
\begin{pmatrix}
g_1&*&*&*&*&*&*\\
&\ddots&*&*&*&*&*\\
&&g_k&*&*&*&*\\
&&&g_0&*&*&*\\
&&&&*&*&*\\
&&&&&\ddots&*\\
&&&&&&*
\end{pmatrix}
&\mapsto(g_1,\dots,g_k,g_0).
\end{alignat*}
\end{itemize}
For $\mathrm {GSO}(m+2, m)$,$\mathrm {O}(m+2, m)$, and $\mathrm {SO}(m+2, m)$, we define $\tau_1\times\tau_2\times\cdots\times\tau_{k}\rtimes\pi_0$ similarly.

When $F$ is a local field, let 
$$\guu^+=\guu^+_{F}=\ker(\epsilon_{E/F}\circ\nu)\subset \guu(F), \ \gll^+_{2}=\gll^+_{2,{F}}=\guu^+_{F}\cap\gll_2(F),$$
where $\ep$ is the quadratic character of $F^{\times}$ which corresponds to $E/F$ (resp. the trivial character of $F^{\times}$)  if $E$ is field (resp. $E=F\times F$).
When $F$ is a global field, let
$$\guu^+_{\aaf}=\guu(\aaf)\cap\prod_v\guu^+_{F_v}, \ \gll_{2,\aaf}^+=\guu^+_{\aaf}\cap\gll_2(\aaf).$$

\subsection{Accidental isomorphism ${\rm PGSO}(4,2)\simeq{\rm PGU}(2,2)$}\label{her-isom}
We introduce the isomorphism $j:{\rm PGU}(2,2)\simeq{\rm PGSO}(4,2)$.
We basically follows \cite{morimoto2014theta}, but there are a few different notations.

Let $$V=\{B(x)\ | \ x=(x_1,x_2,x_3,x_4,x_5,x_6)\in F^{6}\},$$ where
$$B(x)=\begin{pmatrix}
 0&-x_1&x_2&-x_3+x_4\kappa\\
 x_1&0&x_3+x_4\kappa&x_5\\
 -x_2&-x_3-x_4\kappa&0&x_6\\
 x_3-x_4\kappa&-x_5&-x_6&0
 \end{pmatrix}.$$
We define a quadratic form $\langle \cdot,\cdot\rangle_V$ on $V$ by
$$\langle B(x),B(x')\rangle_V={\rm Tr}(B(x)J_4 ^tB(x')^c J_4), \ x=(x_i), \ x'=(x'_i)\in F^6.$$
 Then,
  \begin{align*}\langle B(x),B(x')\rangle_V&=-2(2x_3'x_3-2dx_4x'_4+x_1x'_6+x_6x'_1+x_2x'_5+x_5'x_2)\\
 &=-2xS_2 \!^tx'
 \end{align*}
 so that ${\rm GO}(V,\langle,\rangle_V)\simeq{\rm GO}(4,2)$.
It is known that there is a constant $\alpha\in E_v\baz$ such that $\alpha g B^t\!g\in V$ for each $g\in{\rm GU}(2,2)$ and $B\in V$.
Therefore the map
\begin{align*}{\rm PGU}(2,2)&\map {\rm PGSO}(V,\langle,\rangle_V)\\
 g&\mapsto[B\mapsto \alpha g B^t\!g]
 \end{align*}
is a well-defined.
Furthermore, it is isomorphism according to \cite[2.1]{morimoto2014theta}.
We denote the map 
 $${\rm PGU}(2,2)\bij{\rm PGSO}(V,\langle,\rangle_V)\bij{\rm PGSO}(4,2)$$
 by $j$.
From now on, we often identify objects on ${\rm PGU}(2,2)$ or ${\rm PGSO}(4,2)$ with the pullbacks of them on ${\rm GU}(2,2)$ or ${\rm GSO}(4,2)$.
For example, 
\begin{itemize}
\item an automorphic form of ${\rm PGU}(2,2)(\aaf)$ is regarded as an automorphic form of ${\rm GU}(2,2)(\aaf)$ with trivial central character,
\item a representation of ${\rm GSO}(4,2)(F_v)$ with the trivial central character is regarded as a representation of ${\rm PGSO}(4,2)(F_v)$,
\end{itemize} 

 For the following sections,  we write the images of some elements of $\guuuu$ explicitly: 
 \begin{align*}
 j\left(
 \begin{pmatrix}
 1&&&\\
 &1&&\\
 &&\lambda&\\
 &&&\lambda
 \end{pmatrix}
 \right)&=\begin{pmatrix}
 \lambda^{-1}&&&&&\\
 &1&&&&\\
 &&1&&&\\
 &&&1&&\\
 &&&&1&\\
 &&&&&\lambda
 \end{pmatrix},&&\lambda\in\gll_1,\\
 j\left(
 \begin{pmatrix}
 \alpha&&&\\
 &\beta&&\\
 &&(\beta^c)^{-1}&\\
 &&&(\alpha^c)^{-1}
 \end{pmatrix}
 \right)&=\begin{pmatrix}
N(\alpha\beta)&&&&\\
&N(\alpha)&&&\\
&&\alpha^c\beta&&\\
&&&N(\beta)&\\
&&&&1
 \end{pmatrix},&&\alpha,\beta\in\res\gll_{1},\\
  j\left(
 \begin{pmatrix}
 \alpha&&\\
 &A&\\
 &&(\alpha^c)^{-1}\det A
 \end{pmatrix}
 \right)&=\begin{pmatrix}
 ({\rm N}_{E/F}(\alpha)\det A^{-1})A&&\\
 &\alpha^c&\\
 &&A
 \end{pmatrix},&&\alpha\in\res\gll_{1}, A\in\gll_2.
 \end{align*}
 (We often identify $\alpha=a+b\kappa\in \res\gll_{1}$ with 
 $\left(
\begin{smallmatrix}
 a&db \\
 b&a 
\end{smallmatrix}
\right)\in{\rm GSO}(2,0)$.)
\subsection{Theta lifts}\label{SS: tl}
In this subsection, we introduce theta lifts.
There are some styles  to define the Weil representation.
To avoid confusion, we always adopt the one used in \cite{gan2016gross}, which is a slightly modified version of what is used in \cite{kudla1994splitting}.

 We put $\psi_{E}=\psi_{F}\circ {\rm Tr}_{E/F}$. For $k,m\in\zz_{\geq0}$, let 
$$R_m=\{(g,h)\in\gll_2\times\mathrm {GO}(m+2,m) \ | \ \nu(g)=\nu(h)\}$$
 and
$$R'_k=\{(g,h)\in\guu\times\mathrm {GU}(2k+1) \ | \ \nu(g)=\nu(h)\}.$$

\paragraph{Symplectic-orthogonal case.}
Firstly, we define the local Weil representation and theta lifts.
Let $F$ be a local field.
Then, we define the Weil representation $\omega^m=\omega^m_{\psi_F}$ of $\sll_2\times\mathrm {O}(m+2, m)$ on $\mathcal S(F^{2m+2})$ by
\begin{align*}
\omega^m(1, h)\varphi(x) &=\varphi(h^{-1}x), && h\in \mathrm {O}(m+2, m),\\
\omega^m\left(
\left(
\begin{smallmatrix}
 a&0 \\
 0&a^{-1} 
\end{smallmatrix}
\right)
, 1\right)\varphi(x)&= {\epsilon_{E/F}}(a)|a|_{F}^{m+1}\varphi(ax), && a\in F\baz ,\\
\omega^m\left(
\left(
\begin{smallmatrix}
 1&b \\
 0&1 
\end{smallmatrix}
\right)
, 1\right)\varphi(x)&= \varphi(x)\psi_{F}(\tfrac{1}{2}b{}^t\!xS_mx), && b\in F ,\\
\omega^m\left(
\left(
\begin{smallmatrix}
 0&-1 \\
 1&0 
\end{smallmatrix}
\right)
, 1\right)\varphi(x)
&= \gamma_{E/F}^{-1}\int_{ F^{2m+2}}\varphi(-y')\psi_{F}({}^t\! xS_my')dy'
\end{align*}
for $\varphi\in \mathcal S(F^{2m+2})$ and $x\in F^{2m+2}$.
Here, $dy$ is the selfdual  Haar measure and $\gamma_{E/F}$ is the Weil constant, a 8-th root of unity (see \cite[3]{kudla1994splitting}).
Furthermore, by using the normalized left transition operator $L$ of $\mathrm {GO}(m+2,m)$ on $\mathcal S(F^{2m+2})$ such that
\begin{align*}L(h)\varphi(x)=|\nu(h)|^{-(m+1)/2}_{F}\varphi(h^{-1}x)&&h\in\mathrm {GO}(m+2,m),\end{align*}
we extend the Weil representation $\omega^{m}$ to $R_m$ by
$$\omega^{m}(g,h)=\omega^{m}(g\left(
\begin{smallmatrix}
 1&0 \\
 0&\nu(h)^{-1} 
\end{smallmatrix}
\right),1)\circ L(h)$$
for $(g,h)\in R_k(F)$ (see \cite{roberts1996theta}) and put 
$$\Omega^{m}=\Omega^{m}_{\psi_F}={\rm ind}_{R_m}^{\gll_{2}^+\times\mathrm {GO}(m+2,m)}\omega^{m},$$
where we denote by $\rm ind$ the compact induction.
If $\tau'$ is an irreducible admissible representation  of $\mathrm \sll_2$ (resp. $\gll_{2}^+$), the maximal $\tau'$-isotypic quotient of $\omega^{m}$ (resp. $\Omega^m$) has the form
$$\tau'\boxtimes\Theta^m(\tau'),$$
where $\Theta^m(\tau')$ is a smooth representation of  $\mathrm \sll_2$ (resp. $\gll_{2}^+$).
It is known that $\Theta^m(\tau')$ has finite length (see \cite{kudla1986local} and \cite[Lemma 3.1]{gan2009restrictions}).
Moreover, the maximal semisimple quotient $\theta^m(\tau')$ of $\Theta^m(\tau')$ is irreducible if $\Theta^m(\tau')\neq 0$ (Howe duality, see \cite{howe1989transcending}, \cite{wahawe}, \cite{gan2016proof} and \cite{roberts1996theta}).
We note that the central character of $\Theta^m(\tau')$ is $\chi_{\tau'}^{-1}\epsilon_{E/F}$ for an irreducible representation $\tau'$ of $\gll_{2}^+$, where $\chi_{\tau'}$ is the central character of $\tau'$.
Furthermore, by the conservation relation \cite[Theorem 1.10, Theorem 7.6]{sun2015conservation} and \cite[Lemma 3.1]{gan2009restrictions}, the following holds immediately.
\begin{lemm}\label{l:coms}
Let $\tau'$ be an irreducible admissible representation  of $\gll_{2}^+$ and $m\geq2$.
Then, $\theta^m(\tau')|_{{\rm GSO}(m+2,m)}$ is nonzero and irreducible. 
\end{lemm}
 
Secondary we define the global Weil representation and theta lifts  (we define them only for the similitude case for simplicity). Let $F$ be a number field and let $\omega^{m}=\otimes_{v}\omega^{m}_{F_v}$. For $\varphi\in \mathcal S(\aaf^{2m+2})$ and a cusp form $f$ of $\gll_{2,\aaf}^+$ (i.e. a  left $\gll_{2,F}$-invariant function on $\gll_{2,\aaf}^+$ such that $\smallint_{F\backslash\aaf}f(\left(
\begin{smallmatrix}
 1&b \\
 &1 
\end{smallmatrix}
\right)
g)db=0$ for each $g\in\gll_{2,\aaf}$), let 
\begin{align*}\theta^m_\varphi(g,h)=\sum_{\gamma\in F^{2m+2}}\omega^{m}(g,h)\varphi(\gamma)&&(g,h)\in R_m(\aaf)
\end{align*}
and
\begin{align*}\theta^m_\varphi(f)(h')=\int_{\sll_2(F)\backslash\sll_2(\aaf)}\theta^m_\varphi(g_0g', h')\overline {f(g_0g')}dg_0&&h'\in  \mathrm {GO}(m+2, m)(\aaf),
\end{align*}
where $g'\in \gll_{2,\aaf}^+$ satisfies $\nu(h')=\nu(g')$ (we note that the integral does not depends on the choice of $g'$).
For an irreducible cuspidal representation $\tau'$ of $\gll_{2,\aaf}^+$, let $\theta^m(\tau')$ be the representation of $\mathrm {GO}(m+2, m)(\aaf)$ generated by $\{\theta^m_\varphi(f)\ | \ \varphi\in \mathcal S(\aaf^{2m+2}),f\in\tau'\}$.
We note that $\theta^m(\tau')\simeq \otimes_v\theta^m(\tau'_v)$ if $\theta^m(\tau')$ is square integrable modulo center.

For the next section, we note ``tower property''.
Our claim is for similitude groups, we can proof it similarly to that for isometry groups \cite{rallis1984howe}.
\begin{lemm}\label{l:tow}
Let $m'$ be a non-negative integer such that $m'\leq m$.
We denote by $Q$ the parabolic subgroup of ${\rm GO}(m+2,m)$ such that $B^{GO}_m\subset Q$ and the Levi subgroup $M_Q$ of $Q$ is isomorphic to $\gll_m'\times{\rm GO}(m-m'+2,m-m')$.
Then, for each $\varphi\in \mathcal S(\aaf^{2m+2})$ and cusp form $f$ of $\gll_{2,\aaf}^+$, the constant term 
$\theta_\varphi^m(f)_{N_Q}$ of $\theta_\varphi^m(f)$ along the unipotent radical $N_Q$ of $Q$ satisfies
$$\theta^m_\varphi(f)_{N_Q}(h')=\int_{\sll_2(F)\backslash\sll_2(\aaf)}\sum_{\gamma\in F^{2(m-m')+2}}\int_{\aaf^{m'}}(\omega^m (g_0g',h')\varphi)\left(
\begin{pmatrix}
R\\
\gamma\\
0
\end{pmatrix}
\right)dR \overline{f(g_0g')}dg_0$$
for $h'\in {\rm GO}(m+2,m)(\aaf)$, where $g'$ is taken as above.
Especially,
$$\theta^m_\varphi(f)_{N_Q}(ul)=|\det l_1||\nu (l_0)|^{-m'/2}\int_{\aaf^{m'}}\theta^{m-m'}_{\varphi_R}(f)(l_0)dR$$
for $u\in N_Q(\aaf)$ and $l={\rm diag}(l_1,l_0,*)\in M_Q(\aaf)$ ($l_1\in\gll_{m'}(\aaf),\ l_0\in{\rm GO}(m-m'+2,m-m')(\aaf)$), where
we define $\varphi_R$ by $x\mapsto \varphi\left(\left(
\begin{smallmatrix}
R\\
x\\
0
\end{smallmatrix}
\right)\right)$ for $x\in\aaf^{2(m-m')+2}$
.
\end{lemm}
\paragraph{Unitary group case.}
Firstly let $F$ be a local field and $\rho$ a character of $E\baz$ which satisfies that $\rho|_{F\baz}=\epsilon_{E/F}$.
We define the Weil representation $\omega^{(k,\rho)}=\omega^{(k,\rho)}_{\psi_F}$ of $\uu\times\mathrm {U}(2k+1)$ on $\mathcal S(E^{2k+2})$ by

\begin{align*}
\omega^{(k,\rho)}(1, h)\varphi(x) &=\varphi(h^{-1}x), && h\in \mathrm {U}(2k+1),\\
\omega^{(k,\rho)}\left(
\left(
\begin{smallmatrix}
 a&0 \\
 0&\ (a^{-1})^c 
\end{smallmatrix}
\right)
, 1\right)\varphi(x)&= \rho(a)|a|_{E}^{(2k+1)/2}\varphi(a^cx), && a\in E\baz ,\\
\omega^{(k,\rho)}\left(
\left(
\begin{smallmatrix}
 1&b \\
 0&1 
\end{smallmatrix}
\right)
, 1\right)\varphi(x)&= \varphi(x)\psi_{E}(\tfrac{1}{2}b{}^t\!x^cJ_{2k+1}x), && b\in F,\\
\omega^{(k,\rho)}\left(
\left(
\begin{smallmatrix}
 0&-1 \\
 1&0 
\end{smallmatrix}
\right)
, 1\right)\varphi(x)
&= \gamma_{E/F}^{-1}\int_{E^{2k+1}}\varphi(-y)\psi_{E}({}^t\! x^cJ_{2k+1}y)dy
\end{align*}
for $\varphi\in \mathcal S(E^{2k+2})$ and $x\in E^{2k+2}$.
Here, $dy$ is the selfdual Haar measure and $\gamma_{E/F}$ is the Weil constant which is equal to the one in unitary group case.
We extend $\omega^{(k,\rho)}$ to $R'_k$ and define $\Omega^{(k,\rho)}=\Omega^{(k,\rho)}_{\psi_F}$ similarly to symplectic-orthogonal group case.
Then, for an irreducible admissible representation $\sigma'$ of $\mathrm {U}(2k+1)$ (resp. $\mathrm {GU}(2k+1)$), the maximal $\sigma'$-isotypic quotient of $\omega^{(k,\rho)}$ (resp. $\Omega^{(k,\rho)}$) has the form
$$\Theta^\rho(\sigma')\boxtimes\sigma'$$
where $\Theta^\rho(\sigma')$ is some smooth representation of $\uu$ (resp. $\guu^+$).
In this case, Howe duality is also holds.
We note that the central character of $\Theta^\rho(\sigma')$ is $\chi_{\sigma}^*\rho$ for an irreducible representation $\sigma'$ of $\mathrm {GU}(2k+1)$, where $\chi_{\sigma'}^*=(\chi_{\sigma'}^c)^{-1}$ is the conjugate dual of the central character $\chi_{\sigma'}$ of $\sigma'$.
Moreover, $\Theta^\rho(\sigma')$ satisfies that
$$\Theta^\rho(\sigma')\simeq\chi_{\sigma}^*\rho\boxtimes\Theta^\rho(\sigma'_0),$$
where $\sigma'_0:=\sigma'|_{\mathrm {U}(2k+1)}$ (it is irreducible) and we regard the representation $\chi_{\sigma}^*\rho\boxtimes\Theta^\rho(\sigma'_0)$ of $E\baz\times\mathrm {U}(2k+1)$ as a representation of $\mathrm {GU}(2k+1)$
through the isomorphism 
 \begin{align*}
 E\baz\times\mathrm {U}(2k+1)/\{z^{-1},z1_{2k+1} \ |\ z\in\bu(1)\}&\bij\mathrm{GU}(2k+1),\\
 (\alpha,g)&\mapsto \alpha g.
 \end{align*}

Secondly let $F$ be a number field and $\rho=\otimes_v\rho_v$ an automorphic character of $\aae\baz$ which satisfies that $\rho|_{\aaf\baz}=\ep$ and $\omega^{(k,\rho)}=\otimes_v\omega^{(k,\rho_v)}_{F_v}$.
For $\varphi\in\mathcal S(\aae^{2k+1})$ and a cusp form $f$ of $\mathrm {GU}(2k+1)(\aaf)$, we put

\begin{align*}\theta^\rho_\varphi(g,h)=\sum_{\gamma\in E^{2k+1}}\omega^{(k,\rho)}(g,h)\varphi(\gamma)&&(g,h)\in R_k(\aaf)
\end{align*}
and
\begin{align*}\theta^\rho_\varphi(f)(g')=\int_{\mathrm {U}(2k+1)(F)\backslash\mathrm {U}(2k+1)(\aaf)}\theta^\rho_\varphi(g', h_0h')\overline {f(h_0h')}dh_0&&\mbox{resp.}\ g'\in\guu_{\aaf}^+
\end{align*}
where $h'\in \mathrm {GU}(2k+1)(\aaf)$ satisfies $\nu(h')=\nu(g')$ (we note that the integral does not depends on the choice of $h'$).
For an irreducible cuspidal automorphic representation $\sigma'$ of $\mathrm {GU}(2k+1)(\aaf)$, let $\theta^\rho(\sigma')$ be the representation of $\guu_{\aaf}^+$ generated by $\{\theta^\rho_\varphi(f)\ | \ \varphi\in \mathcal S(\aae^{2k+1}),f\in\sigma'\}$.
We note that $\theta^\rho(\sigma')\simeq \otimes_v\theta^{\rho_v}(\sigma'_v)$ if $\theta^\rho(\sigma')$ is square integrable modulo center.

If we regard ${\rm GU}(2k+1)$ as a subgroup of $\mathrm {GO}(2k+2, 2k)$ by the identification $F^{4k+2}=E^{2k+1}$, then
$$\omega^{(k,\rho)}|_{R_{2k}\cap R'_{k}(F)}=\omega^{2k}|_{R_{2k}\cap R'_{k}(F)}$$
 for any local field $F$.
 Moreover, for irreducible admissible representations $\tau'$ and $\sigma'$ of $\gll_2^+$ and ${\rm GU}(2k+1)$ the following identity  (local see-saw identity)
 $${\rm Hom}_{\gll_2^+}(\Theta^\rho(\sigma'), \tau')\simeq{\rm Hom}_{{\rm GU}(k+1,k)}(\Theta^{2k}(\tau'), \sigma')$$
holds. 
If $F$ is global, there is a similar equation (global see-saw identity), namely
$$\int_{\mathrm {GU}(2k+1)(F)\aaf\baz\backslash\mathrm {GU}(2k+1)(\aaf)}\theta^\rho(f_1)(h)\overline{f_2(h)}dh=\int_{\gll^+_{2,F}\aaf\baz\backslash \gll^+_{2,\aaf}}\theta^{2k}(f_2)(g)\overline{f_1(g)}dg,$$
where $f_1$ is a cusp form of $ \gll^+_{2,\aaf}$ with central character $\chi_1$ and $f_2$ is a cusp form of $\mathrm {GU}(2k+1)(\aaf)$ with with central character $\chi_2$ such that $\chi_2|_{\aaf\baz}=\chi_1^{-1}\ep$.

\subsection{Determination of near equivalence class}
We keep the notations of \S\ref{her-ov}.
For each finite place $v$ of $F$, let $\varpi_v$ be a prime element of $F_v$ and $q_v$ be the cardinality of residue field of $F_v$.
In this subsection, we show the following proposition.
\begin{prop} \label{P:near}
$\mathcal H(\tau)\subset L^2_{\psi}$.
\end{prop}
Since A-parameters are determined by their near equivalence classes, what we need to do is to determine (almost all) unramified components of $\mathcal H(\tau)$ explicitly so we show the following lemma.
\begin{lemm}\label{L:unram}
We assume that $v$ splits or is unramified on $E$ and that $\tau_v$ and $\chi_v$ and $\theta(\tilde \tau_v)$ are unramified.
Let $\chi_v\chi'_1\rtimes 1_{\bu(0)(F_v)}$ be a induced representation which $\tau_v$ is an irreducible subrepresentation of, where $\chi_1'$ is a some unramified character of $E_v\baz$.
Then, $\theta(\tilde \tau_v)|_{{\rm GSO}(4,2)(F_v)} \circ j$ is an irreducible subquotient of 
$$\chi_1'^*|\cdot|_{E_v}^{-\frac{1}{2}}\times\chi_1'|\cdot|_{E_v}^{-\frac{1}{2}}\rtimes |\cdot|_{F_v}$$
as a representation of $\guuuu(F_v)$ with the trivial central character.
\end{lemm}
\begin{proof}
First of all, $\theta(\tilde \tau_v)|_{{\rm GSO}(4,2)(F_v)} \circ j$ is irreducible by Lemma \ref{l:coms}.
Furthermore, by the formulae in \S\ref{her-isom}, the equation
$$\chi_1\times \chi_2 \rtimes \chi_0\circ j=(\chi_1\chi_2\circ N\cdot\chi_0^c)\times(\chi_1\circ N\cdot\chi_0)\rtimes\chi_1^{-1}$$
holds as a representation of $\guuuu(F_v)$ for any characters $\chi_1,\chi_2$ of $F_v\baz$ and $\chi_0$ of $E_v\baz\simeq {\rm GSO}(2,0)(F_v)$.
%

We only think about the case that $v$ splits on $E$
(if $v$ is unramified on $E$, the proof is similar). 
Then $\chi_v\chi'_1\rtimes 1_{\bu(0)(F_v)}$ is irreducible because of Ramanujan bound, so $\tau_v=\chi_v\chi'_1\rtimes 1_{\bu(0)(F_v)}$.
Since the central character $\chi_v\chi'_1|_{\bu(1)(F_v)}$ of $\tau_v$ is equal to $\chi_v|_{\bu(1)(F_v)}$, we can write $\chi'_1=(\chi',\chi')$ for some (unramified) character $\chi'$ of $F_v\baz$.
Then, we have
$$\tilde\tau_v=\ind^{\gll_{2}(F_v)}_{B(F_v)}\chi'^2\boxtimes\chi'^{-1},$$
where $B$ is the subgroup of upper triangular matrices of $\gll_2$ and $\chi'^2\boxtimes\chi'^{-1}$ is regarded a character of $B_v$ by
$$\chi'^2\boxtimes\chi'^{-1}((\begin{smallmatrix}
a&*\\
& a^{-1}b
\end{smallmatrix}))=\chi'^2(a)\chi'^{-1}(b).$$

For $s\in\cc$ and $\varphi\in\mathcal S(F^6_v)$, we define $Z(s,\varphi)$ by the following integral
$$Z(s,\varphi)=\int_{F_v\baz}\varphi(^t(b,0,\dots,0))|x|^s dx.$$
The above integral is holomorphic if $\varphi(0)=0$.
Further, for $\varphi_0\in\mathcal S(F^6_v)$ such that
$$\varphi(^t(b_1,\dots,b_6))=\begin{cases}
1 \ &\mbox{if $|b_i|\leq1$ for all $i$};\\
0 \ &\mbox{otherwise},
\end{cases}$$
we have
$$Z(s,\varphi_0)=C\frac{1}{1-q_v^{s}}$$
for $s\in\cc$ such that $\mathrm{Re}s>0$, where $C\neq 0$ is a some constant.
Thus $Z(s,\varphi)$ have always has the meromorphic continuation to all $\cc$ which is holomorphic at any $s\neq0$.
$Z(s,\varphi)$ satisfies the following equation
$$Z(s,\omega^2(g',h')\varphi)=|a|^{-s+3}_{F_v}|a'|_{F_v}^{s} |\nu(\alpha)|^{-\frac{3}{2}}_{F_v}Z(s,\varphi)=\delta_{B}(g')^{\frac{1}{2}}\delta_{B_2^{GO}}(h')^{\frac{1}{2}}|a|^{-s+2}_{F_v}|a'|_{F_v}^{s-2}|b'|_{F_v}^{-1}|\nu(\alpha)|^{\frac{1}{2}}_{F_v}Z(s,\varphi)$$
for all $$(g',h')=\left(
\begin{pmatrix}
a&*\\
&N(\alpha) \cdot a^{-1}
\end{pmatrix},
\begin{pmatrix}
a'&*&*&*&*\\
&b'&*&*&*\\
&&\alpha&*&*\\
&&&*&*\\
&&&&*
\end{pmatrix}
\right)\in (B\times B^{GO}_2)\cap R_2'(F_v)$$
and $\varphi\in\mathcal S(F_v^6)$, where $\delta_{B}$ and $\delta_{B_2^{GO}}$ are the modulus characters of ${B}(F_v)$ and ${B_2^{GO}}(F_v)$ respectively.
Taking $s\neq0$ such that $\chi'^{2}=|\cdot|^{-s+2}_{F_v}$ (recall Ramanujan bound), we have a nonzero $R_{2}(F_v)'$-map
\begin{align*}\omega^2&\map \ind^{R_{2}(F_v)}_{B_{R}(F_v)}((\chi'^2\boxtimes\chi'^{-1})\boxtimes(\chi'^{-2}\boxtimes|\cdot|_{F_v}^{-1}\boxtimes\chi'|\cdot|_{F_v}^{\frac{1}{2}}\circ \nu))|_{M_R(F_v)},\\
\varphi &\mapsto [(g,h)\mapsto Z(s,\omega^2(g,h)\varphi)],
\end{align*}
where $B_R=(B\times B^{GO}_2)\cap R_2$ and $M_R$ is its the Levi subgroup.
We can see that the map
\begin{align*}\ind^{\gll_{2}(F_v)}_{B(F_v)}(\chi'^2\boxtimes\chi'^{-1})\boxtimes(\chi'^{-2}\times|\cdot|_{F_v}^{-1}\rtimes\chi'|\cdot|_{F_v}^{\frac{1}{2}}\circ \nu)
&\map
\ind^{R_{2}(F_v)}_{B_{R}(F_v)}(\chi'^2\boxtimes\chi'^{-1})\boxtimes(\chi'^{-2}\boxtimes|\cdot|_{F_v}^{-1}\boxtimes\chi'|\cdot|_{F_v}^{\frac{1}{2}}\circ \nu)|_{M_R(F_v)}\\
f&\mapsto f|_{R_2'(F_v)}
\end{align*}
is bijective $R_2(F_v)$-map.
Thus we have a nonzero $\gll_{2}(F_v)\times \mathrm{GO}(4,2)(F_v)$-map
$$\Omega^2\map\ind^{\gll_{2}(F_v)}_{B(F_v)}(\chi'^2\boxtimes\chi'^{-1})\boxtimes(\chi'^{-2}\times|\cdot|_{F_v}^{-1}\rtimes\chi'|\cdot|_{F_v}^{\frac{1}{2}}\circ \nu)$$
by Frobenius reciprocity, so $\theta(\tilde \tau_v)|_{{\rm GSO}(4,2)(F_v)} \circ j$ is a subquotient of
$$\chi'^{-2}\times|\cdot|_{F_v}^{-1}\rtimes(\chi'|\cdot|_{F_v}^{\frac{1}{2}}\circ N) \circ j=\chi_1'^*|\cdot|_{E_v}^{-\frac{1}{2}}\times\chi_1'|\cdot|_{E_v}^{-\frac{1}{2}}\rtimes |\cdot|_{F_v}.$$
\end{proof}
There is a natural nonzero $\uuuu(F_v)$-map 
$$\chi_1'^*|\cdot|_{E_v}^{-\frac{1}{2}}\times\chi_1'|\cdot|_{E_v}^{-\frac{1}{2}}\rtimes |\cdot|_{F_v}\map\chi_1'^*|\cdot|_{E_v}^{-\frac{1}{2}}\times\chi_1'|\cdot|_{E_v}^{-\frac{1}{2}}\rtimes 1_{\bu(0)}$$
and the irreducible unramified subquotient of $\chi_1'^*|\cdot|_{E_v}^{-\frac{1}{2}}\times\chi_1'|\cdot|_{E_v}^{-\frac{1}{2}}\rtimes 1_{\bu(0)}$ is equal to the unramified representation which corresponds to $\phi_{\psi_v}$, so this lemma leads Proposition \ref{P:near} immediately.
\begin{cor}\label{C:iso}
$$\mathcal H(\tau)\simeq \otimes_v(\theta^2(\tilde \tau_v)|_{\mathrm{GSO}(4,2)(F_v)}\circ j)|_{\uuuu(F_v)}.$$
\end{cor}
\begin{proof}
$\theta^2(\tilde \tau)$ is square integrable modulo center as already mentioned, so $\theta^2(\tilde \tau)\simeq\otimes_v\theta^2(\tilde \tau_v)$.
Therefore we have to show the surjective map
$$\theta^2(\tilde \tau)\surj \mathcal H(\tau)$$
is injective.
We fix an irreducible subrepresentation $\pi$ of $\theta^2(\tilde \tau)|_{{\rm GSO}(4,2)(\aaf)}\circ j|_{\uuuu(\aaf)}$ (as abstract representation) whose image in $\mathcal H(\tau)$ is nonzero.
For a place $v$ of $F$, the set $S$ of irreducible subrepresentation of $\theta^2(\tilde \tau_v)|_{\mathrm{GSO}(4,2)(F_v)}\circ j|_{\uuuu(F_v)}$ contain at most two distinct elements and the action of $\guuuu(F_v)$ on $S$ is transitive since $\theta^2(\tilde \tau_v)|_{\mathrm{GSO}(4,2)(F_v)}\circ j$ is irreducible and
$${\rm GU}(n,n)(F_v)/E_v^\times{\rm U}(n,n)(F_v)\simeq\{\pm 1\} \mbox{ or } \{1\}.$$
Then, $\pi_v\in S$ obviously.
When $v$ is nonarchimedean, we have $\pi_v\in\Pi_{\psi_v}$ by Proposition \ref{P:near}.
Moreover, $\pi_v$ is fixed by the action of $\guuuu(F_v)$ by lemma \ref{L:fix}, so $\#S=1$, namely $\theta^2(\tilde \tau_v)|_{\mathrm{GSO}(4,2)(F_v)}\circ j|_{\uuuu(F_v)}$ is irreducible if $v$ is nonarchimedean.
Thus $\theta^2(\tilde \tau)|_{{\rm GSO}(4,2)(\aaf)}\circ j|_{\uuuu(\aaf)}$ is multiplicity free and any irreducible subrepresentation $\pi'$ of it satisfies $\pi_v\simeq\pi'_v$ for each nonarchimedean place $v$ and $\pi_v^{g_v}\simeq\pi'_v$ for each archimedean place $v$, where $g_v\in \guuuu(F_v)\ ({\rm mod}\ E_v\baz\uuuu(F_v))$.
Then, by weak approximation theory, $\pi^g=\pi'$ for some $g\in\guuuu(F)$.
Since $\mathcal H(\tau)$ is stable (as an automorphic representation) by the action of $\guuuu(F)$, the image of $\pi'$ in $\mathcal H(\tau)$ is nonzero.
Consequently, the map $\theta^2(\tilde \tau)$ to $\mathcal H(\tau)$ is injective.
\end{proof}
\subsection{The proof of Theorem \ref{T: her}}\label{det}
In this subsection, we prove Theorem \ref{T: her}.
To prove it, the following proposition plays an essential role.
\begin{prop}\label{T:locikeda}
Let $v$ be a place of $F$ and $\tau'=\pi((\phi_\chi)_v,\eta')\in\Pi_{(\phi_\chi)_v}$ and we put $\tilde\tau'=(\tau'\boxtimes\chi_v)|_{\gll^+_{2,v}}$.
Then, we have
$$(\theta^2(\widetilde\tau')|_{\mathrm{GSO}(4,2)(F_v)}\circ j)|_{\uuuu(F_v)}=\pi(\psi_v,\eta').$$
\end{prop}
We show this proposition at the bottom of this section.
It leads Theorem \ref{T: her} as follows.
\begin{proof}[Proof of Theorem \ref{T: her}]
By Lemma 4.5 and Proposition \ref{T:locikeda}, we have $\mathcal H (\tau)\simeq I_2(\tau)$ abstractly.
Therefore we have to show that $L_\psi^2$ is multiplicity free, but we have seen it (Corollary \ref{c:mf}).
\end{proof}

To prove Theorem \ref{T:locikeda}, we use the following three lemmas.
\begin{lemm}\label{pac}
Each subrepresentation of $(\theta^2(\widetilde\tau')|_{\mathrm{GSO}(4,2)(F_v)}\circ j)|_{\uuuu(F_v)}$ is an element of $\Pi_{\psi_v}$.
\end{lemm}
\begin{proof}
Since we can take a globalization of $\tau'$ by Lemma \ref{L:glza}, Corollary \ref{C:iso} leads this lemma immediately.
\end{proof}

%
%

In the remaining two lemmas we assume that $\phi_v$ is of type 2-1 or 2-2 i.e. $\widehat{\overline{\mathcal S}}_{(\phi_\chi)_v}\simeq\widehat{\overline{\mathcal S}}_{\psi_v} \simeq\{\pm1\}$.
If $\phi_v$ is of such type, then we put $\tau_{\pm}=\pi((\phi_{\chi})_v,\pm1)$ for  short.

\begin{lemm}\label{L:first}
Let $v$ be a nonarchimedean place of $F$ which does not split in $E$.
We assume that $\phi_v$ is of type 2-1 or 2-2 i.e. $\phi_v=\mu_1\oplus\mu_1^{-1}$ for some character $\mu_1$ of $E_v\baz$ such that $\mu_1|_{F_v\baz}=1$ ($\phi_v$ is of type 2-2 if and only if $\mu_1=\mu_1^{-1}$).
Then,
 $\theta^0(\tilde \tau_{+})=\mu_1$ holds.
\end{lemm}
\begin{proof}
By the the following see-saw

$$
\xymatrix{
   \mathrm {GO}(2,0)   \ar@{-}[d]   & &   \guu^+ \ar@{-}[d]     \\
   \mathrm{GU}(1)   \ar@{-}[urr]     & &   \gll_2^+,\ar@{-}[ull]   \\
}
$$
the equation
$$\dim_\cc\mathrm{Hom}_{\mathrm{GU}(1)(F_v)}(\Theta^0(\tilde \tau_{+}), \mu_1)=\dim_\cc\mathrm{Hom}_{\gll_{2,F_v}^+}(\Theta^{\chi_v\mu_1^{-1}}(\mu_1), \tau_+)$$ holds.
By \cite[Proposition 4.5]{gan2016gross},
$$\tau_{+}\boxtimes\chi_v=\theta^{\chi_v\mu_1^{-1}}(\mu_1)$$
so that
\begin{align*}
\dim_\cc\mathrm{Hom}_{\guu_{F_v}^+}(\Theta^{\chi_v\mu_1^{-1}}(\mu_1), \tau_{+}\boxtimes\chi_v)
=1
\end{align*}
Since $\guu^+_{F_v}=E_v\baz\gll_{2,F_v}$ and the central character of $\Theta^{\chi_v\mu_1^{-1}}(\mu_1)$ is $\chi_v$,
$$\dim_\cc\mathrm{Hom}_{\gll_{2,F_v}^+}(\Theta^{\chi_v\mu_1^{-1}}(\mu_1), \tau_+)=\dim_\cc\mathrm{Hom}_{\guu_{F_v}^+}(\Theta^{\chi_v\mu_1^{-1}}(\mu_1), \tau_{+}\boxtimes\chi_v)$$
holds.
Thus $\theta^0(\tilde \tau_{+})=\mu_1$ holds.
\end{proof}

In the last lemma we assume that $v$ is archimedean.
Before stating it, we prepare some notations.
We put
$$G=\left\{g\in \gll_4(\cc)\left| ^tg^c
\left(\begin{smallmatrix}
1_2&0\\
0&-1_2
\end{smallmatrix}\right)
g = 
\left(\begin{smallmatrix}
1_2&0\\
0&-1_2
\end{smallmatrix}\right)\right. \right\}$$
and we define an involution $\vartheta$ on $G$ by $\vartheta(g)=(^tg^c)^{-1}$.
Let
$$K=\left\{g\in G\ | \ g= \vartheta(g)\right\} , \ T=\{\mathrm{diag}(e^{i\theta_1},\cdots,e^{i\theta_4})\ |\ \theta_i\in\rr\}\subset K.$$
Then $K$ is a maximal compact subgroup of $G$ and $T$ is its maximal compact torus.
If $v$ is a place of $F$ such that $E_v/F_v=\cc/\rr$, we identify $G$ and $\bu(4)(F_v)$ by the following isomorphism

\begin{align*}
G &\bij \bu(4)(F_v)\\
g&\mapsto g_0g g_0^{-1},
\end{align*}
where 
\begin{align*}g_0=\sqrt{2}^{-1}
\begin{pmatrix}
1&&&-1\\
&1&-1&\\
&1&1&\\
1&&&1
\end{pmatrix}\in\gll_2(\rr).
\end{align*}
We denote by $\mathfrak g_0, \mathfrak k_0, \mathfrak t_0$ the Lie algebras of $G, K, T$, respectively.
Let $\mathfrak g, \mathfrak k, \mathfrak t$ be the complexifications of $\mathfrak g_0, \mathfrak k_0, \mathfrak t_0$, respectively. 
We define $e_j, f_j\in i\mathfrak t_0^*\subset\mathfrak t^*$ ($j=1,2$) by
\begin{align*}
e_j:\mathrm{diag}(a_1,\cdots.a_4)&\mapsto a_i\\
 f_j:\mathrm{diag}(a_1,\cdots.a_4)&\mapsto a_{2+j}.
 \end{align*}
Let $\mathfrak g=\mathfrak k + \mathfrak p$ be a Cartan decomposition. Then, the roots of $\mathfrak k$ and $\mathfrak p$ relative to  $\mathfrak  t$ are
$$\Delta(\mathfrak k,\mathfrak t)=\{\pm(e_1-e_2), \pm(f_1-f_2)\}, \ \Delta(\mathfrak p)=\{\pm(e_i-f_j), | i,j\in\{1,2\}\},$$
respectively.
We put
$$\Delta_{l,+}=\{\pm(e_1-f_2),\pm(e_2-f_1)\}, \ \Delta_{l,-}=\{\pm(e_1-e_2),\pm(f_1-f_2)\}$$
and  
$$\Delta_{v,+}=\{(e_1-e_2),-(f_1-f_2),(e_1-f_1),-(e_2-f_2)\}, \Delta_{v,-,\sigma}=\{\sigma(e_i-f_j) \ | \ i,j\in\{1,2\}\}$$
for $\sigma=\pm1$.
We define $\vartheta$-stable parabolic subalgebras $\mathfrak q^+$ and $\mathfrak q^-_\sigma$ ($\sigma=\pm$) of $\mathfrak g$ by
$$\mathfrak q^+=\mathfrak l^+\oplus\mathfrak v^+, \ \mathfrak q^-_{\sigma}=\mathfrak l^-\oplus\mathfrak v^-_{\sigma},$$
where $\mathfrak l^{+}, \ \mathfrak l^{-}$ are Levi subalgebras of $\mathfrak q^+$ and $\mathfrak q^-_\sigma$ defined by
$$\mathfrak l^{+}=\mathfrak t\oplus\bigoplus_{\alpha\in\Delta_{l,+}}\mathfrak g_\alpha,\ \mathfrak l^{-}=\mathfrak t\oplus\bigoplus_{\alpha\in\Delta_{l,-}}\mathfrak g_\alpha,\ $$
and $\mathfrak v^{+}$, $\mathfrak v^{-}_{\sigma}$ are nilradicals of $\mathfrak q^+$ and $\mathfrak q^-_\sigma$ defined by
$$ \mathfrak v^{-}_{\sigma}=\bigoplus_{\alpha\in\Delta_{v,-,\sigma}}\mathfrak g_\alpha, \  \mathfrak v^{+}=\bigoplus_{\alpha\in\Delta_{v,+}}\mathfrak g_\alpha.$$

Then, by \cite{MR2019}, the following holds.
\begin{lemm}\label{mr}
Let $v$ be an archimedean place of $F$ such that $E_v/F_v=\cc/\rr$.
We assume that $\phi_v$ is of type 2-1 or 2-2 i.e. $\phi_v=\chi_{2k}\oplus\chi_{-2k}$ for some $k\in\zz_{\geq0}$  ($\phi_v$ is of type 2-2 if and only if $k=0$), where $\chi_t \ (t\in\zz)$ is defined by $\chi_t(r e^{i\theta})=e^{it\theta}$ for $r\in\rr_{>0}, \theta\in\rr/2\pi \zz$. Then, the elements of A-packet which corresponds to $\psi_v=\chi_{2k}[2]\oplus\chi_{-2k}[2]$ are described by using cohomological induction (in \cite{knapp1995cohomological}):
\begin{align*}\pi(\psi_v, 1)&=A_{\mathfrak q^+}(\lambda^+_k),\\
\pi(\psi_v,-1)&=\bigoplus_{\sigma=\pm 1}A_{\mathfrak q^-_{\sigma}}(\lambda^-_{k,\sigma}),
\end{align*}
where
\begin{align*}\lambda^+_k&=(k-1)(e_1-e_2-f_1+f_2),\\
\lambda^-_{k,\sigma}&=\sigma(k-1)(e_1+e_2-f_1-f_2).
\end{align*}
\end{lemm}
\begin{rem}
\begin{itemize}
\item Since $\lambda^+_k$ and $\lambda^-_{k,\sigma}$ are in weakly fair range, $\pi(\psi_v, \pm1)$ are unitary.
Moreover, each component of $\pi(\psi_v, \pm1)$ above is irreducible if $k\geq 1$ since $\lambda^+_k$ and $\lambda^-_{k,\sigma}$ are in good range in such case.
Though $\lambda^+_0$ and $\lambda^-_{0,\sigma}$ is not in weakly good range, each component of $\pi(2\chi_0[2], \pm1)$ is zero or irreducible as mentioned in \cite[p.10]{MR2019}.
We recall that $\pi(\psi',1)\neq0$ for any $\psi'$.
\item By Blattner formula,
the highest weights of $K$-types of $A_{\mathfrak q^+}(\lambda^+_k)$ are contained in
$$S^+_k=k(e_1-e_2-f_1+f_2)+\sum_{\alpha\in \Delta(\mathfrak p)\cap \Delta_{v,+}}\zz_{\geq0}\alpha$$
and the highest weights of $K$-types of $A_{\mathfrak q^-_{\sigma}}(\lambda^-_{k,\sigma})$ are contained in
$$S^-_{k,\sigma}=\sigma\left((k+1)(e_1+e_2-f_1-f_2)+\sum_{\alpha\in  \Delta(\mathfrak p)\cap\Delta_{v,-,+}}\zz_{\geq0}\alpha\right).$$
\end{itemize}
\end{rem}

Then, the proof of Theorem \ref{T: her} is as follows.
\begin{proof}[Proof of Theorem \ref{T: her}]
We proof this theorem by three cases:
\begin{enumerate}
\renewcommand{\labelenumi}{(\arabic{enumi})}
\item $\Pi_{\psi_v}$ is singleton i.e. $v$ splits in $E$ or $\phi_v$ is of type 1-1, 1-2, or 3,
\item $v$ is nonarchimedean and $\phi_v$ is of type 2-1 or 2-2,
\item $v$ is archimedean and $\phi_v$ is of type 2-1 or 2-2.
\end{enumerate}
Let us see them from the top.

We assume (1).
Then, by Lemma \ref{pac}, it is sufficient to show that the unique element of $\Pi_{\psi_v}$, which is described as an induced representation, is irreducible.
If $v$ is split or $\phi_v$ is of type 3, then it is led by the Ramanujan bound.
If $\phi_v$ is of type 1-1, 1-2, then this it holds by Lemma \ref{L:fr}.

We assume (2).
Then, $\theta^2(\tilde \tau_{\pm})|_{\uuuu(F_v)}$ is irreducible by Lemma \ref{L:fix} and by Lemma \ref{pac}.
Moreover,
\begin{itemize}
\item if $\theta^2(\tilde \tau_{-})|_{\uuuu(F_v)}\neq\pi(\psi_v,1)$, then  $\theta^2(\tilde \tau_{-})|_{\uuuu(F_v)}=\pi(\psi_v,-1)$ and
\item if $\theta^2(\tilde \tau_{+})|_{\uuuu(F_v)}\neq\theta^2(\tilde \tau_{-})|_{\uuuu(F_v)}$, then  $\{\theta^2(\tilde \tau_{\pm})|_{\uuuu(F_v)}\}=\{\pi(\psi_v,\pm1)\}.$
\end{itemize}
Thus it is sufficient to show that $\theta^2(\tilde \tau_{-})|_{\uuuu(F_v)}\neq\pi(\psi_v,1)$ and $\theta^2(\tilde \tau_{+})|_{\uuuu(F_v)}\neq\theta^2(\tilde \tau_{-})|_{\uuuu(F_v)}.$
By Lemma \ref{L:first} and the conservation relation \cite[Theorem 1.10]{sun2015conservation} and \cite[Lemma 3.1]{gan2009restrictions},  we have $\theta^1(\tilde\tau_{-})=0$ and $\theta^2(\tilde\tau_{-})\neq0$. 

First, let $\phi_v$ is type 2-1. 
Then $\Pi_{(\phi_\chi)_v}$ is a supercuspidal L-packet.
Therefore $\theta^2(\tilde\tau_{-})$ is supercuspidal and $\theta^2(\tilde\tau_{+})$ is not since $\theta^1(\tilde\tau_{-})=0$ and Lemma \ref{L:first}.
Especially $\theta^2(\tilde \tau_{+})|_{\uuuu(F_v)}\neq\theta^2(\tilde \tau_{-})|_{\uuuu(F_v)}$.
Since the element of $\Pi_{\phi_{\psi_v}}$ is not supercuspidal, so that $\theta^2(\tilde \tau_{-})|_{\uuuu(F_v)}\neq\pi(\psi_v,1)$.

Second, let $\phi_v=2\mu_1$ is type 2-2. Then, $\Pi_{(\phi_\chi)_v}$ is equal to the set of irreducible subrepresentation of 
$$\ind^{\uu(F_v)}_{B_2(F_v)}\mu_1\chi_v.$$
Therefore $\{\tau^0_{\pm}=\tau_{\pm}|_{\sll_2(F_v)}\}$ is the L-packet of $\sll_2(F_v)$ with L-parameter $\phi_0=2\epsilon_{E_v/F_v}\oplus \bf1$ since 
$$\ind^{\uu(F_v)}_{B_2(F_v)}\mu_1\chi_v|_{\sll_2(F_v)}=\ind^{\sll_2(F_v)}_{(B_2\cap \sll_2)(F_v)}\epsilon_{E_v/F_v}.$$
By \cite[Lemma 3.1]{gan2009restrictions},
$$\theta^2(\tau^0_{\pm})|_{\mathrm {SO} (4, 2)(F_v)}=\theta^2(\tilde \tau_{\pm})|_{\mathrm {SO} (4, 2)(F_v)}.$$
By \cite[Theorem 4.1]{atobe2017local}, $\theta^2(\tau^0_{-})|_{\mathrm {SO} (4, 2)(F_v)}$ is the unique quotient of 
$${\rm St}(|\cdot|_{F_v}^{\frac{1}{2}}) \rtimes 1_{{\rm SO}(2,0)}.$$
Here, St$(\chi')$ is the unique subrepresentation of 
$$(\chi'\circ\det)\otimes\ind^{\gll_2}_B|\cdot|_{F_v}^{\frac{1}{2}}\boxtimes|\cdot|_{F_v}^{-\frac{1}{2}}$$
  ($B$ is the the subgroup of $\gll_2$ which consists of upper triangular matrices) for a character $\chi'$ of $F_v\baz$.
By the Frobenius reciprocity, $\theta^2(\tilde \tau_{-})|_{\mathrm {GSO} (4, 2)(F_v)}$ is the unique irreducible quotient of 
$${\rm St}(|\cdot|_{F_v}^{\frac{1}{2}}) \rtimes |\cdot|^{-\frac{1}{2}}_{E_v}\zeta$$
for some character $\zeta$ of $E_v\baz$ which satisfies $\zeta|_{F_v\baz}=1$.
By the formulae in \S\ref{her-isom}, it is easy to verify that
$${\rm St}(|\cdot|_{F_v}^{\frac{1}{2}}) \rtimes |\cdot|^{-\frac{1}{2}}_{E_v}\zeta\circ j=|\cdot|^{\frac{1}{2}}_{E_v}\zeta\rtimes ({\rm St}(1)\boxtimes \zeta)$$
as a representation of $\guuuu(F_v)$,
so $\theta^2(\tilde \tau_{- 1})|_{\uuuu(F_v)}$ is the unique quotient of 
$$|\cdot|^{\frac{1}{2}}_{E_v}\zeta\rtimes  ({\rm St}(1)\boxtimes \zeta|_{\uu(F_v)}),$$
so the L-parameter of $\theta^2(\tilde \tau_{-})|_{\uuuu(F_v)}$ is $\zeta|\cdot|^{\frac{1}{2}}_{F_v}\oplus \zeta[2]\oplus \zeta^*|\cdot|_{F_v}^{-\frac{1}{2}}$.
Thus we have 
$$\theta^2(\tilde \tau_{-})|_{\uuuu(F_v)}\neq\pi(\psi_v,1).$$
Similarly, $\theta^2(\tilde \tau_{+})|_{\uuuu(F_v)}$ has L-parameter $2\xi|\cdot|^{\frac{1}{2}}_{F_v}\oplus 2\xi^*|\cdot|^{-\frac{1}{2}}_{F_v}$ for some $\xi$.
Especially, $\theta^2(\tilde \tau_{+})|_{\uuuu(F_v)}\neq\theta^2(\tilde \tau_{-})|_{\uuuu(F_v)}$.

We assume (3).
Then, it is sufficient to show the following:
\begin{enumerate}
\renewcommand{\labelenumi}{(\roman{enumi})}
\item the irreducible representation of $K$ whose highest weight is $\mu_k^+=k(e_1-e_2-f_1+f_2)$ is a $K$-type of $(\theta^2(\widetilde\tau_+)|_{\mathrm{GSO}(4,2)(F_v)}\circ j)|_{\uuuu(F_v)}$,
\item for $\sigma=\pm$, the irreducible representation of $K$ whose highest weight is $\mu_{k,\sigma}^-=\sigma k(e_1-e_2-f_1+f_2)$ is a $K$-type of $(\theta^2(\widetilde\tau_-)|_{\mathrm{GSO}(4,2)(F_v)}\circ j)|_{\uuuu(F_v)}$.
\end{enumerate}
Actually, since $\mu_{k,\pm}^-\not\in S_{k}^+,$ and $\mu_{k,\pm}^-\in S_{k,\pm}^-$, $\pi(\psi_v,-1)$ is a subrepresentation of $(\theta^2(\widetilde\tau_-)|_{\mathrm{GSO}(4,2)(F_v)}\circ j)|_{\uuuu(F_v)}$ and $A_{\mathfrak q^-_{\sigma}}(\lambda^-_{k,\sigma})\neq0$ if (i) is true.
$(\theta^2(\widetilde\tau_-)|_{\mathrm{GSO}(4,2)(F_v)}\circ j)|_{\uuuu(F_v)}$ is a sum of at most two distinct irreducible representations, so $\pi(\psi_v,-1)=(\theta^2(\widetilde\tau_-)|_{\mathrm{GSO}(4,2)(F_v)}\circ j)|_{\uuuu(F_v)}$.
Similarly, we have $\pi(\psi_v,1)=(\theta^2(\widetilde\tau_+)|_{\mathrm{GSO}(4,2)(F_v)}\circ j)|_{\uuuu(F_v)}$ if (ii) is true.

We denote by $K'$ the maximal compact subgroup of ${\rm GO}(4,2)(F_v)$ (and ${\rm O}(4,2)(F_v)$) which contains (the pullback of) $j^{-1}(K)$.
Firstly we can see that $\tilde\tau_\pm|_{\sll_2(F_v)}=D_{\pm{(2k+1)}}$, where $D_t$ is the (limit of) discrete series whose minimal 
$\bu(1)(\rr)$-type is $(t):e^{i\theta}\mapsto e^{it\theta}$ for $t\in\zz$.
Since $(\pm(2k+1))$ is of minimal degree in $\bu(1)(\rr)$-types of $D_{\pm{(2k+1)}}$, there is a $K'$-type of 
$\theta^2(D_{\pm{(2k+1)}})=\theta^2(\tilde \tau_{\pm})|_{{\rm O}(4,2)(F_v)}$
 which correspond to $(\pm(2k+1))$ and we can compute them explicitly by \cite[Proposition 4]{PAUL2005270}.
Since ${\rm GO}(4,2)(F_v)=F_v\baz{\rm O}(4,2)(F_v)$ and the central character of $\theta^2(\tilde \tau_{\pm})$ is trivial, we have (i) and (ii) by simple computation using the formulae in \S\ref{her-isom}.
(we note that $(\mu_k^+)^g=\mu_k^+$ and  $(\mu_{k,+}^-)^g=\mu_{k,-}^-$
for $g\in\guuuu(F_v)\setminus E_v\baz \uuuu(F_v)$). 

\end{proof}

\section{The proof of the main theorem}\label{S:mt}
Let $\pi$, $\chi$, $\phi$, $\phi_{\chi}$, $\psi$ be as in \S \ref{S:hm}.
 Let $\gamma$ be an automorphic character of $\bu(1)(\aaf)$ such that $\phi\neq \check\gamma\boxplus\check\gamma^{-1}$. 
Following \S\ref{miyaa}, we define $\mathcal M_{\psi}(\check\gamma)=\phi\boxplus\check\gamma^{-1}$.
Furthermore, let $\mathcal M_{I_2(\tau)}(\gamma)$ be the Miyawaki lift of $\gamma$ with respect to $I_2(\tau)$ and the following diagonal embedding
\begin{align*}i:\bu(1)\times \bu(3) &\inj \uuuu\\
(\alpha,h)&\mapsto \delta\begin{pmatrix}
\alpha&\\
&h
\end{pmatrix}\delta^{-1},
\end{align*}
where 
$$\delta=\begin{pmatrix}
 0&-\kappa&0&0\\
 \kappa/2&0&\kappa/2&0\\
 1&0&-1&0\\
 0&0&0&1
 \end{pmatrix}.$$
As already remarked in Remark \ref{emb}, choice of embedding is not essential to define $\mathcal M_{I_2(\tau)}(\gamma)$.

The aim of this section is to show the following proposition. 

\begin{prop}\label{P:pre}
$\mathcal M_{I_2(\tau)}(\gamma)$ is contained in $L^2_{\mathcal M_{\psi}(\check\gamma)}.$ Moreover,
 any irreducible subrepresentation $\sigma$ of $\mathcal M_{I_2(\tau)}(\gamma)$ satisfies that $\theta^{\chi_v\check\gamma_v^{-1}}(\sigma_v\otimes\gamma_v\circ\det)=\tau_v$ for all $v$.
\end{prop}
 
Though there is a gap between the main theorem (Theorem \ref{T:main}) and the above proposition, the following results bridge it immediately:
\begin{itemize}
\item Proposition 4.5 of \cite{gan2016gross}, which gives the theta correspondence of unitary groups over $p$-adic fields for almost equal rank cases i.e. the cases that the difference of dimensions of underling spaces of two unitary groups is one.
\item Theorem 3.4 of \cite{paul2000howe}, which gives the theta correspondence of unitary groups over $\rr$ for almost equal rank cases. 
\item Theoreme 1 of \cite{minguez2008correspondance}, which gives the theta correspondence of general linear groups over $p$-adic fields.
\item III.9 (resp. Proposition 2.2) of \cite{moeglin1989correspondance} (resp. \cite{howe2comp}), which gives the theta correspondence of general linear groups over $\rr$ (resp. $\cc$). 
\end{itemize}

\subsection{Square integrability}
Firstly, we show the square-integrability of $\mathcal M_{I_2(\tau)}(\gamma)$.

Let $\mathcal F=\theta^2_{\varphi}(f)\circ j|_{\uuuu(\aaf)}\in I_2(\tau) \ (f\in\tilde\tau, \varphi\in\mathcal S(\aaf^{6}))$.
Let $P_0=M_0U_0=B_4, P_1=M_1U_1,$ and $P_2=M_2U_2$ be the parabolic subgroups containing $B_4$ such that the Levi subgroup $M_i$ of $P_i$ is isomorphic to $\res\gll_{1}\times\res\gll_{1}, \ \res\gll_{2},$ and $\gll_{1}\times \uu,$ respectively.
For $i=0,1,2$, let $pr_i$ be the projection of $P_i$ on $M_i$.
We note that the parabolic subgroup $Q_i=L_iN_i$ $(i=1,2)$ of $\go$ containing $B_2^{GO}$ such that Levi subgroup $L_i$ of $Q_i$ is isomorphic to $\gll_{3-i}\times{\rm GO}(4-i,2-i)$ corresponds to $P_i$ via $j$.

 We define the constant term $\mathcal F_i$ of $\mathcal F$ along $U_i$ by 
$$\mathcal F_i(g)=\int_{U_i(F)\backslash U_i(\aaf)}\mathcal F(ug)du, \ g\in \uuuu(\aaf)$$ and put
$$s\mathcal F=\mathcal F-\mathcal F_2-\mathcal F_1+\mathcal F_0.$$
Then,
$$\mathcal F=s\mathcal F+(\mathcal F_1-\mathcal F_0)+\mathcal F_2.$$
We fix a Siegel set $S_4$ of $\uuuu(\aaf)$ (see \cite[I.2.1]{moeglin_waldspurger_1995}) and regard $s\mathcal F,$ $\mathcal F_1-\mathcal F_0,$ and $\mathcal F_2$ as functions on $S_4$.
Then $s\mathcal F$ is bounded (see \cite[I.2.12]{moeglin_waldspurger_1995}), so all we have to do is is to evaluate the contribution of  $\mathcal F_1-\mathcal F_0$ and $\mathcal F_2$
.
\begin{lemm}
We have
$\int_{\bu(1)(F)\backslash\bu(1)(\aaf)}\mathcal F_2(i(\alpha,h))\overline{\gamma(\alpha)}d\alpha=0$ for $h\in \bu(3)(\aaf)$.
\end{lemm}
\begin{proof}
Since the contribution of $h$ can be absorbed into $\varphi$, we can assume $h=1$.
For $\alpha\in\bu(1)(\aaf)$, 
$$\delta\begin{pmatrix}
 \alpha&\\
 &1
 \end{pmatrix}\delta^{-1}=\begin{pmatrix}
 1&&&\\
 &(1+\alpha)/2&-\kappa(1-\alpha)/4&\\
 &-(1-\alpha)/\kappa&(1+\alpha)/2&\\
 &&&1
 \end{pmatrix}$$
so the above integral is well-defined.
Furthermore, for $\beta\in\aae\baz$ such that $\beta/\beta^c=\alpha$,
$$j(i(\alpha,1))=
 \begin{pmatrix}
 a&-db/2&&&\\
 -2b&a&&&\\
 &&\beta&&\\
 &&&a&db/2\\
 &&&2b&a
 \end{pmatrix}.$$
 Here, we identify $\beta=a+b\kappa\in \aae\baz$ and 
 $\left(
\begin{smallmatrix}
 a&db \\
 b&a 
\end{smallmatrix}
\right)\in{\rm GSO}(2,0)(\aaf)$.
Then, by Lemma \ref{l:tow}, we have
\begin{align*}\mathcal F_2\left(i(\alpha,1)\right)&=\theta^2_{\varphi}(f)_{N_2}(j(i(\alpha,1)))\\
&=\int_{ \aaf^2}\theta^0_{\varphi_R}(f)(\beta)dR.
\end{align*}

 If $\theta^0(\tilde \tau)=0$, this lemma holds.
We assume $\theta^0(\tilde \tau)\neq0$.
 Then, it is easy to verify that $\phi$ is of type 2 and $\theta^0(\tilde \tau)|_{{\rm GSO}(2,0)(\aaf)}=\mu_1\oplus\mu_1^{-1}$ by the see-saw in Lemma 4.4. 
Since $\gamma(\alpha)=\check\gamma(\beta)$ and we assumed $\check \gamma\neq\mu_1^{\pm1}$, so this lemma holds in general.
\end{proof}

\begin{lemm}
$\mathcal F_1-\mathcal F_0$ is bounded.
\end{lemm}
\begin{proof}
%

Let $K_4$ be a maximal compact subgroup of $\uuuu(\aaf)$ such that $B_4(\aaf)K_4=\bu(2,2)(\aaf)$.
Because of the finiteness of $\varphi$ for some maximal compact subgroup of $\go(\aaf)$,
there is a finite subset $\{\varphi_j\}_j$ of $\mathcal S(\aaf^{6})$ and a finite set $\{c_j\}_j$ of bounded function of $K_{4}$, we have
$$\mathcal F (gk)=\sum_j c_j(k)\theta^2_{\varphi_j}(f)(j(g))$$
for $g\in \uuuu(\aaf)$ and $k\in K_{4}$.
Then, by Lemma \ref{l:tow} and the formulae in \S\ref{her-isom}, we have
$$(\mathcal F_1-\mathcal F_0)(ulk)=\sum_j c_j(k)\int_{\aaf^2}s\theta^1_{\varphi_{jR}}(f)(j\left(\left(
\begin{smallmatrix}
\alpha&*\\
&\beta
\end{smallmatrix}\right)\right))dR,$$
for $u\in U_2(\aaf)$, $k\in K_{4}$, and
$$ \ l=
\begin{pmatrix}
\alpha&*&&\\
&\beta&&\\
&&(\beta^c)^{-1}&*\\
&&&(\alpha^c)^{-1}
\end{pmatrix}
\in P_0\cap M_1(\aaf)  \ (\alpha,\beta\in\aae\baz),$$ where $sf'$ is the constant term of an automorphic form $f'$ of $\gll_2(\aae)$ along some Borel subgroup of $\res\gll_2$.
We note that the integral of right hand side of above equation is finite sum, namely there are finite constants $\{a_i\}_i$ and finite subset $\{R_i\}_i$ of $\aaf^2$ such that
$$\int_{\aaf^2}s\theta^1_{\varphi_{jR}}(f)(j\left(\left(
\begin{smallmatrix}
\alpha&*\\
&\beta
\end{smallmatrix}\right)\right))dR=\sum_ia_i s\theta^1_{\varphi_{jR_i}}(f)(j\left(\left(
\begin{smallmatrix}
\alpha&*\\
&\beta
\end{smallmatrix}\right)\right)).$$
Thus $\mathcal F_1-\mathcal F_0$ is bounded.
\end{proof}
By above lemmas, $\mathcal M_{\mathcal F}(\gamma)$ is bounded.
Thus the following holds by Proposition \ref{P:ee}.
\begin{cor}\label{C:sq}
$\mathcal M_{I_2(\tau)}(\gamma)\subset L^2_{\mathcal M_{\psi}(\check\gamma)}.$
\end{cor}

\subsection{See-Saw}
Secondly, we  describe $\mathcal M_{I_2(\tau)}(\gamma)$ in terms of global theta lifts.

We embed $\guuu$ into $\gso$ induced the following identification
\begin{align*}F^{6}&\simeq E^3\\
 \begin{pmatrix}
 x\\
 y\\
 a\\
 b\\
 z\\
 w
 \end{pmatrix}
 &\mapsto
  \begin{pmatrix}
 x/\kappa-y/2\\
 a+b\kappa\\
 -z-w\kappa/2
 \end{pmatrix}.
\end{align*}
Furthermore, we put
\begin{align*}e:\mathrm{GU}(3)   &\inj  \mathrm{GU}(2,2)\\
g&\mapsto i(\det g/\nu(g),g).
\end{align*}
Then, we have the following lemma by easy computation.
\begin{lemm}\label{surp}
The following diagram commutes:
$$
\xymatrix{
   \mathrm {GSO}(4,2)   \ar@{<-^{)}}[d] \ar@{=}[r]  &   \mathrm {GSO}(4,2) \ar@{->>}[r] &   \mathrm {PGSO}(4,2)  \ar@{->}[d]^{j}_\simeq   \\
   \mathrm{GU}(3)   \ar@{^{(}->}[r]^{e} &  \mathrm{GU}(2,2) \ar@{->>}[r]    &  \mathrm{PGU}(2,2) .\\
}
$$
\end{lemm}
Let $\sigma$ be an irreducible subrepresentation of $L^2_{\mathcal M_{\psi}(\check\gamma)}$.
We note that the central character of $\sigma$ is $\gamma^{-1}$ and $\sigma$ is cuspidal. 
\begin{prop}\label{long}
$\sigma\subset\mathcal M_{I_2(\tau)}(\gamma)\Leftrightarrow\theta^{\chi\check\gamma^{-1}}(\sigma\otimes\gamma\circ\det)=\tau$
\end{prop}
\begin{proof}
Let $\Gamma$ be an automorphic character of $\aae\baz$ such that $\Gamma|_{\bu(1)(\aaf)}=\gamma$ and we extend $f\in\sigma$ on $\guuu(\aaf)=\uuu(\aaf)\aae\baz$ by
$$f(\alpha g)=\Gamma(\alpha)^{-1}f(g)$$
for $\alpha\in\aae\baz$ and $g\in\uuu(\aaf)$.
 and let $\Sigma=\sigma\boxtimes\Gamma^{-1}$, an extension of $\sigma$ on $\guuu(\aaf)$. 
Similarly, we identify $f'\in\tau$ with the extension of it on $\guu_{\aaf}^{+}$ by $\chi$.
We note that the map $f'\mapsto f'|_{\gll_{2,\aaf}^+}$ is the natural bijection between $\tau$ and $\tilde\tau.$
We take $f\in \sigma$, $\varphi\in\mathcal S(\aae^3)$, $f'\in\tau$ and we put $\mathcal F=\theta^2_{\varphi}(f'|_{\gll_{2,\aaf}^+})\circ j|_{\uuuu(\aaf)}$.

Firstly we have
\begin{align*}
\int_{\uu(F)\backslash\uu(\aaf)}\theta^{\chi\check\gamma^{-1}}_{\varphi}(f\cdot\gamma(\det))(g)\overline {f'(g)}dg
&=\int_{\guu_F^{+}\aaf\baz\backslash\guu_{\aaf}^{+}}\theta^{\chi\check\gamma^{-1}}_{\varphi}(f\cdot\Gamma(\det/\nu))(g)\overline {f'(g)}dg\\
&=\int_{\gll_{2,F}^{+}\aaf\baz\backslash\gll_{2,\aaf}^{+}}\theta^{\chi\check\gamma^{-1}}_{\varphi}(f\cdot\Gamma(\det/\nu))(g)\overline {f'(g)}dg.
\end{align*}
By the following see-saw
$$
\xymatrix{
   \mathrm {GO}(4,2)   \ar@{-}[d]   & &   \guu^+ \ar@{-}[d]     \\
   \mathrm{GU}(3)   \ar@{-}[urr]     & &   \gll_2^+,\ar@{-}[ull]   \\
}
$$
we have
\begin{align*}
\int_{\gll_{2,F}^{+}\aaf\baz\backslash\gll_{2,\aaf}^{+}}\theta^{\chi\check\gamma^{-1}}_{\varphi}(f\cdot\Gamma(\det/\nu))(g)\overline {f'(g)}dg
=&\int_{\guuu(F)\aaf\baz\backslash\guuu(\aaf)}\theta^{2}_{\varphi}(f'|_{\gll_{2,\aaf}^+})(h)\overline {f(h)\Gamma(\det h/\nu(h))}dh.
\end{align*}
Decomposing the variable $h$ of the above integral into the following product $h=\beta h_0$ ($h_0\in\uuuu(\aaf),$ $\beta\in \aae\baz$)
we have
\begin{align*}
&\int_{\guuu(F)\aaf\baz\backslash\guuu(\aaf)}\theta^{2}_{\varphi}(f'|_{\gll_{2,\aaf}^+})(h)\overline {f(h)\Gamma(\det h/\nu(h))}dh\\
=&\int_{\bu(3)(F)\backslash\uuu(\aaf)}\int_{ E\baz\aaf\baz\backslash\aae\baz}\mathcal F(i(\beta (\beta^c)^{-1} \det h_0,h_0))\overline {f(h_0)\gamma(\beta (\beta^c)^{-1} \det h_0)}d\beta dh_0
\end{align*}
by Lemma \ref{surp}.
Finally, putting $\alpha=\beta (\beta^c)^{-1}\in \bu(1)(\aaf)$, we have
\begin{align*}
&\int_{\bu(3)(F)\backslash\uuu(\aaf)}\int_{ E\baz\aaf\baz\backslash\aae\baz}\mathcal F(i(\beta (\beta^c)^{-1} \det h_0,h_0))\overline {f(h_0)\gamma(\beta (\beta^c)^{-1} \det h_0)}d\beta dh_0\\
=&\int_{\bu(3)(F)\backslash\uuu(\aaf)}\int_{ \bu(1)(F)\backslash\bu(1)(\aaf)}\mathcal F(i(\alpha \det h_0,h_0))\overline {f(h_0)\gamma(\alpha  \det h_0)}d\alpha dh_0\\
=&\int_{\bu(3)(F)\backslash\uuu(\aaf)}\int_{ \bu(1)(F)\backslash\bu(1)(\aaf)}\mathcal F(i(\alpha,h_0))\overline {\gamma(\alpha)f(h_0)}d\alpha dh_0\\
=&\int_{\bu(3)(F)\backslash\uuu(\aaf)}\mathcal M_ {\mathcal F}(h_0)\overline {f(h_0)}d\alpha dh_0,
\end{align*}
so
$$\int_{\uu(F)\backslash\uu(\aaf)}\theta^{\chi\check\gamma^{-1}}_{\varphi}(f\cdot\gamma(\det))(g)\overline {f'(g)}dg
=\int_{\bu(3)(F)\backslash\uuu(\aaf)}\mathcal M_ {\mathcal F}(h_0)\overline {f(h_0)}d\alpha dh_0$$
holds.
The equation means that
$$\sigma\subset\mathcal M_{I_2(\tau)}(\gamma)\Leftrightarrow\theta^{\chi\check\gamma^{-1}}(\sigma\otimes\gamma\circ\det)\supset\tau.$$
Since $\theta^{\chi\check\gamma^{-1}}(\sigma\otimes\gamma\circ\det)$ is irreducible or zero, this proposition holds.
\end{proof}
\subsection{Reduction to local theta lifts}
Lastly, we reduce the nonvanishing of global theta lifts to local theta lifts by using Yamana's result\cite{yamana2014functions}.

\begin{prop}\label{P:thet}
For an irreducible subrepresentation $\sigma'$ of $L^2_{\mathcal M_{\psi}(\check\gamma)\otimes\check\gamma\circ\det}$, if
$\theta^{\chi_v\check\gamma_v^{-1}}(\sigma'_v)\neq0$ for all $v$ then $\theta^{\chi\check\gamma^{-1}}(\sigma')\neq0$.
\end{prop}
\begin{proof}
By \cite[Theorem 10.1]{yamana2014functions}, when $L(\sigma',s)$ has a pole at $s=0$ and has no pole at $s\in \frac{1}{2}\zz_{<0}$, where $L(\sigma'',s)$ is the doubling L-function of $\sigma''$.
Since the functional equation
$$L(\sigma',s)=\epsilon(\sigma',s)L(\sigma'^\vee,1-s)$$
holds, what we need to show is that $L(\sigma'^\vee,s)$ has a pole at $s=1$ and has no pole at $s\in 1+\frac{1}{2}\zz_{>0}$.
Let $S_{\rm fin}$ (resp. $S_\infty$) be the set of finite (resp. infinite) places $v$ of $F$ where $E_v\neq F_v\oplus F_v$ and $\sigma'_v$ is not unramified.
Then, we have
\begin{align*}
L(\sigma'^\vee,s)=L^{GJ}(\phi\otimes\check\gamma^{-1}\circ \det,s)\zeta_E(s)\prod_{v\in S_{\rm fin}}\frac {L(\sigma'^\vee_v,s)(1-q_{E_v}^{-s})}{L^{GJ}(\phi_v\otimes\check\gamma_v^{-1},s)}\prod_{v\in S_\infty}\frac {L(\sigma'^\vee_v,s)}{L^{GJ}(\phi_v\otimes\check\gamma_v^{-1},s)2(2\pi)^{-s}\Gamma(s)},
\end{align*}
where $L^{GJ}(\psi',s)$ is the Godement-Jaquet L-function of automorphic representation $\psi'$ of a general linear group, $\zeta_E(s)$ is the complete Dedekind Zeta function of $E$, and $q_{E_v}$ is the cardinality of the residue field of $E_v$.
It is known that
\begin{itemize}
\item $L^{GJ}(\phi\otimes\check\gamma^{-1}\circ \det,s)$ is holomorphic on ${\rm Re}(s)>1$ and nonzero at $s=1$ (\cite{jacquet1976non}) since $\phi\neq\check\gamma\boxplus\check\gamma^{-1}$,
\item $\zeta_E(s)$ is holomorphic on ${\rm Re}(s)>1$ and has simple pole at $s=1$,
\item for $v\in S_{\rm fin}\sqcup S_\infty$, $L(\sigma'^\vee_v,s)$ is holomorphic on ${\rm Re}(s)>1$ and nonzero at $s=1$ since it is meromorphic function with no zeros,
\item  for $v\in S_{\rm fin}$, $L^{GJ}(\phi_v\otimes\check\gamma_v^{-1},s)$ has no zero on ${\rm Re}(s)>1$ and no pole at $s=1$ since it is a finite product of some factors $(1- q_{E_v}^{t-s})^{-1}$ where $|{\rm Re}t|\leq1/2$,
\item  for $v\in S_{\infty}$, $L^{GJ}(\phi_v\otimes\check\gamma_v^{-1},s)$ has no zero on ${\rm Re}(s)>1$ and no pole at $s=1$ since it is a finite product of some factors $2(2\pi)^{-(s+\lambda+k/2)}\Gamma(s+\lambda+k/2)$ where $|{\rm Re}\lambda|\leq 1/2$ and $k\in\zz_{\geq0}$.
\end{itemize}
Therefore $L(\sigma'^\vee,s)$ has a pole at $s=1$ and has no pole at $s\in 1+\frac{1}{2}\zz_{>0}$.
\end{proof}
By  Corollary \ref{C:sq} and Proposition \ref{P:thet}, Proposition \ref{P:pre} holds.
As mentioned at the beginning of the section, we have finished proving Theorem \ref{T:main}.
  \bibliography{aaaaa}
\bibliographystyle{amsalpha} 

\end{document}